\documentclass[a4paper,oneside,10pt, reqno] {amsart}
\usepackage{graphicx}
\usepackage{pstricks}
\usepackage{amscd}
\usepackage{amsmath}
\usepackage{amsxtra}

\usepackage{amsfonts,amssymb,amsmath,amsthm}
\usepackage{url}
\usepackage{enumerate}

\newcommand{\strutstretchdef}{\newcommand{\strutstretch}{\vphantom{\raisebox{1pt}{$\big($}\raisebox{-1pt}{$\big($}}}}
\theoremstyle{plain}
\newtheorem{theorem}{Theorem}[section]
\newtheorem{lemma}[theorem]{Lemma}
\newtheorem{proposition}[theorem]{Proposition}
\newtheorem{corollary}[theorem]{Corollary}

\theoremstyle{definition}
\newtheorem{definition}[theorem]{Definition}
\newtheorem{problem}[theorem]{Problem}

\theoremstyle{remark}
\newtheorem{remark}[theorem]{Remark}

\numberwithin{equation}{section}

\newlength{\struh}
\newlength{\textminustop}
\strutstretchdef
\newcommand*{\sib}[1]{\mathsf{sib}(#1)}

\newcommand*{\child}[1]{\mathsf{Chi}(#1)}
\newcommand*{\childn}[2]{{\mathsf{Chi}}^{\langle#1\rangle}(#2)}

\newcommand*{\parentn}[2]{{\mathsf{par}}^{\langle#1\rangle}(#2)}

\newcommand*{\Ge}{\geqslant}
\newcommand*{\lambdab}{\boldsymbol\lambda}

\newcommand*{\Le}{\leqslant}
\newcommand*{\parent}[1]{\mathsf{par}(#1)}

\newcommand*{\rootb}{{\mathsf{root}}}

\usepackage{mathrsfs}
\usepackage{epsfig}
\usepackage{ifthen}
\usepackage[all]{xy}

\newcommand{\ncom}{\newcommand}
\ncom{\bq}{\begin{equation}}
\ncom{\eq}{\end{equation}}
\ncom{\beqn}{\begin{eqnarray*}}
\ncom{\eeqn}{\end{eqnarray*}}
\ncom{\beq}{\begin{eqnarray}}
\ncom{\eeq}{\end{eqnarray}}
\ncom{\nno}{\nonumber}
\ncom{\rar}{\rightarrow}
\ncom{\Rar}{\Rightarrow}
\ncom{\noin}{\noindent}
\ncom{\bc}{\begin{centre}}
\ncom{\ec}{\end{centre}}
\ncom{\sz}{\scriptsize}
\ncom{\rf}{\ref}
\ncom{\sgm}{\sigma}
\ncom{\Sgm}{\Sigma}
\ncom{\dt}{\delta}
\ncom{\Dt}{Delta}
\ncom{\lmd}{\lambda}
\ncom{\Lmd}{\Lambda}
\ncom{\eps}{\epsilon}
\ncom{\pcc}{\stackrel{P}{>}}
\ncom{\dist}{{\rm\,dist}}
\ncom{\sspan}{{\rm\,span}}
\ncom{\im}{{\rm Im\,}}
\ncom{\sgn}{{\rm sgn\,}}
\ncom{\ba}{\begin{array}}
\ncom{\ea}{\end{array}}
\ncom{\eop}{\hfill{{\rule{2.5mm}{2.5mm}}}}
\ncom{\eoe}{\hfill{{\rule{1.5mm}{1.5mm}}}}
\ncom{\eof}{\hfill{{\rule{1.5mm}{1.5mm}}}}
\ncom{\hone}{\mbox{\hspace{1em}}}
\ncom{\htwo}{\mbox{\hspace{2em}}}
\ncom{\hthree}{\mbox{\hspace{3em}}}
\ncom{\hfour}{\mbox{\hspace{4em}}}
\ncom{\hsev}{\mbox{\hspace{7em}}}
\ncom{\vone}{\vskip 2ex}
\ncom{\vtwo}{\vskip 4ex}
\ncom{\vonee}{\vskip 1.5ex}
\ncom{\vthree}{\vskip 6ex}
\ncom{\vfour}{\vspace*{8ex}}
\ncom{\norm}{\|\;\;\|}
\ncom{\integ}[4]{\int_{#1}^{#2}\,{#3}\,d{#4}}
\ncom{\inp}[2]{\langle{#1},\,{#2} \rangle}
\ncom{\Inp}[2]{\Langle{#1},\,{#2} \Langle}
\ncom{\vspan}[1]{{{\rm\,span}\#1 \}}}
\ncom{\dm}[1]{\displaystyle {#1}}

\begin{document}
\title[Dirichlet Spaces associated with Directed Trees]{Dirichlet Spaces associated with \\ Locally Finite Rooted Directed Trees}

%\author[S. Chavan and R. Curto]{Sameer Chavan and Ra$\acute{\mbox{u}}$l Curto}
%\address{Indian Institute of Technology Kanpur\\
%Kanpur- 208016, India}
%\email{chavan@iitk.ac.in}
%\address{University of Iowa\\
%Iowa City, Iowa 52242, USA}
%\email{raul-curto@uiowa.edu}
%\thanks{The second named author was partially supported by NSF Grant DMS-1302666.} \ 
\author{Sameer Chavan \and Deepak Kumar Pradhan \and Shailesh Trivedi
}

\address{Department of Mathematics and Statistics\\
Indian Institute of Technology Kanpur, India}
   \email{chavan@iitk.ac.in}

   \email{dpradhan@iitk.ac.in}

 \address{School of Mathematics \\ Harish-Chandra Research Institute \\ 
HBNI, Chhatnag Road, Jhunsi, Allahabad 211019, India}
\email{shaileshtrivedi@hri.res.in}

\thanks{The work of the second author is supported through the NBHM Research Fellowship}



\begin{abstract}
Let $\mathscr T=(V, \mathcal E)$ be a leafless, locally finite rooted directed tree.
We associate with $\mathscr T$ a one parameter family of Dirichlet spaces $\mathscr H_q~(q \Ge 1)$, which turn out to be Hilbert spaces of vector-valued holomorphic functions defined on the unit disc $\mathbb D$ in the complex plane. These spaces can be realized as reproducing kernel Hilbert spaces associated with the positive definite kernel
\beqn
\kappa_{\mathscr H_q}(z, w) &=& \sum_{n=0}^{\infty}\frac{(1)_n}{(q)_n}\,{z^n \overline{w}^n} ~P_{\langle e_\rootb \rangle} \\  &+& 
\sum_{v \in V_{\prec}} \sum_{n=0}^{\infty}  \frac{(n_v +2)_n}{(n_v + q+1)_n}\, {z^n \overline{w}^n}~P_{v}~(z, w \in \mathbb D),
\eeqn
where $V_{\prec}$ denotes the set of branching vertices of $\mathscr T$, $n_v$ denotes the depth of $v \in V$ in $\mathscr T,$ and $P_{\langle e_\rootb \rangle}$,  $~P_{v}~(v \in V_{\prec})$ are certain orthogonal projections. 
%We also discuss some structural properties of the operator $\mathscr M_{z, q}$ of multiplication by $z$ on $\mathscr H_q.$
Further, we discuss the question of unitary equivalence of operators $\mathscr M^{(1)}_z$ and $\mathscr M^{(2)}_z$ of multiplication by $z$ on Dirichlet spaces $\mathscr H_q$ associated with directed trees $\mathscr T_1$ and $\mathscr T_2$ respectively. 
\end{abstract}

\maketitle

%\setcounter{tocdepth}{1}
%\tableofcontents

%\section{Preliminaries}
\section{Introduction} 

The present work is a sequel to \cite{CT}. In that paper, a rich interplay between the directed trees and analytic kernels of finite bandwidth has been exploited to study the weighted shifts on directed trees. The analysis therein was based on Shimorin's analytic model as introduced in \cite{Sh} (for an alternate approach to the function theory of weighted shifts on directed trees, the reader is referred to \cite{BDP}, \cite{BDPP}).
Our work is also motivated partly by the classification theorem \cite[Theorem 9.9]{ACJS} obtained for $2$-isometric weighted shifts on certain directed trees.

The objective of the present paper is to introduce Dirichlet spaces associated with certain rooted directed trees. This is carried out by introducing the so-called Dirichlet shifts on directed trees with weights being certain functions of depth of vertices, and thereafter applying Shimorin's construction \cite{Sh} to these shifts. These spaces can be thought of as vector-valued weighted Dirichlet spaces (cf. \cite{O}). We also discuss the spaces Cauchy dual to Dirichlet spaces. These turn out to be vector-valued Bergman spaces, which play a key role in answering the question of unitary equivalence of Dirichlet shifts associated with two directed trees. 
We collect below some preliminaries required to define the Dirichlet shifts. For a detailed exposition on weighted shifts on directed trees, the reader is referred to \cite{Jablonski} and \cite{CT}.

%\subsection{Preliminaries}

A {\it directed graph} is a pair $\mathscr T= (V,\mathcal E)$, where  $V$ is a nonempty set and $\mathcal E$ is a nonempty subset of $V \times V \setminus \{(v,v): v \in V\}$. An element of $V$ (resp. $\mathcal E$) is called a {\it vertex} (resp. an {\it edge}) of $\mathscr T$.  A finite sequence $\{v_i\}_{i=1}^n$ of distinct vertices is said to be a {\it circuit} in $\mathscr T$ if $n \geqslant  2$, $(v_i,v_{i+1}) \in \mathcal E$ for all $1 \leqslant i \leqslant n-1$ and $(v_n,v_1) \in \mathcal E$. 
We say that two distinct vertices $u$ and $v$ of $\mathscr T$ are {\it connected by a path} if there exists a finite sequence $\{v_i\}_{i=1}^n$ of distinct vertices of $\mathscr T$ $(n \geqslant  2)$ such that $v_1=u$, $v_n=v$ and $(v_i,v_{i+1})$ or $(v_{i+1},v_i) \in \mathcal E$ for all $1 \leqslant i \leqslant n-1$.
A directed graph $\mathscr T$ is said to be {\it connected} if any two distinct vertices of $\mathscr T$ can be connected by a path in $\mathscr T.$  For a
subset $W$ of $V$, define
$$\child{W} := \bigcup_{u\in W} \{v\in V
\colon (u,v) \in \mathcal E\}.$$ 
We define inductively $\childn{n}{W}$ \index{$\childn{n}{W}$} for 
$n \in \mathbb N$ as follows: 
\beqn
\childn{n}{W}:= 
\begin{cases} W & ~\mbox{if }~n=0, \\
\child{\childn{n-1}{W}} & ~\mbox{if~} n \geqslant 
1. \end{cases}
\eeqn
Given $v\in V$, we write $\child{v}:=\child{\{v\}}$.
%$\childn{n}{v}:=\childn{n}{\{v\}}$. 
An element of $\child{v}$ is called a {\it child} of $v.$ 
%The {\it descendants} of a vertex $v \in V$ \index{$\mathsf{Des}(v)$} is given by
%\beqn
%\mathsf{Des}(v):=\bigcup_{n=0}^{\infty} \childn{n}{v}.
%\eeqn
For a given vertex $v \in V,$ consider the set $\mathsf{Par}(v):=\{u \in V : (u, v) \in \mathcal E\}$. 
\index{$\mathsf{Par}(v)$}
If $\mathsf{Par}(v)$ is singleton, then the unique vertex in $\mathsf{Par}(v)$ is called the {\it parent} of $v$, 
\index{$\parent{v}$}
which we denote by $\parent{v}.$
Let the subset $\mathsf{Root}(\mathscr T)$ 
\index{$\mathsf{Root}(\mathscr T)$}
of $V$ be defined as
$$\mathsf{Root}(\mathscr T) := \{v \in V : \mathsf{Par}(v)  = \emptyset\}.$$
%Let the subset $\boldsymbol{\mathsf{Root}}(\mathscr T)$ of $V$ is defined as
%$$\mathsf{Root}(\mathscr T) := \{v \in V : \mathsf{Par}(v) = \emptyset\}.$$
Then an element of $\mathsf{Root}(\mathscr T)$ is called a {\it root} of $\mathscr T$. If $\mathsf{Root}(\mathscr T)$ is singleton, then its unique element 
is denoted  by $\mathsf{root}$. 
We set $V^\circ:=V \setminus \mathsf{Root}(\mathscr T)$. \index{$V^{\circ}$}
A directed graph $\mathscr T= (V,\mathcal E)$ is called a {\it directed tree} if 
$\mathscr T$ has no circuits, $\mathscr T$ is connected and
each vertex $v \in V^\circ$ has a unique parent. 
%\begin{remark}
%It is well-known that every directed tree has at most one root \cite[Proposition 2.1.1]{JJS} (see Figure 1.1). 
%\end{remark}
A directed tree $\mathscr T$ is said to be
\begin{enumerate}
\item[(i)] {\it rooted} if it has a unique root.
\item[(ii)] 
{\it locally finite} if $\mbox{card}(\child u)$ is finite for all $u \in V,$ where  $\mbox{card}(X)$ stands for the cardinality of the set $X.$
\item[(iii)]
{\it leafless} if every vertex has at least one child.
\end{enumerate}

Let $\mathscr T=(V, \mathcal E)$ be a rooted directed tree with root $\rootb$. 
Let $V_{\prec}$ be the set $\{u\in V: \mbox{card}(\mathsf{Chi}(u)) \Ge 2\}$ of branching 
vertices of $\mathscr T$. 
For each $u \in V$, the {\it depth} of $u$ is the unique non-negative integer $n_u$ 
such that
$u \in \mathsf{Chi}^{\langle n_u\rangle}(\mathsf{root})$.
Define the {\it branching index} of $\mathscr 
T$ as 
\[k_\mathscr{T}:=\begin{cases}
 1+\sup\{n_w:w\in V_{\prec}\}& \text{if $V_{\prec}$ is non-empty},\\
 0& \text{if $V_{\prec}$ is empty}.
\end{cases}
\]
We say that $\mathscr T$ is of {\it finite branching index} if $k_\mathscr 
T < \infty.$ We further say that two directed trees are {\it isomorphic} if there exists a bijection between their sets of vertices which preserves directed edges.
%In particular,  we follow the notations set in this paper.

Let $\mathscr T=(V, \mathcal E)$ be a rooted directed tree with root $\rootb$. We always assume that 
$\mbox{card}(V)=\aleph_0.$
%and  $V^{\circ}:=V \setminus \{\rootb\}.$
In what follows, $l^2(V)$ stands for the Hilbert
space of square summable complex functions on $V$
equipped with the standard inner product. Note that
the set $\{e_u\}_{u\in V}$ is an
orthonormal basis of $l^2(V)$, where $e_u \in l^2(V)$
is the indicator function of $\{u\}$. Given a system
$\lambdab = \{\lambda_v\}_{v\in V^{\circ}}$ of non-negative real numbers, 
we define the {\em weighted shift operator} $S_{\lambda}$ on ${\mathscr T}$
with weights $\lambdab$ by
   \begin{align*}
   \begin{aligned}
{\mathscr D}(S_{\lambda}) & := \{f \in l^2(V) \colon
\varLambda_{\mathscr T} f \in l^2(V)\},
   \\
S_{\lambda} f & := \varLambda_{\mathscr T} f, \quad f \in {\mathscr
D}(S_{\lambda}),
   \end{aligned}
   \end{align*}
where $\varLambda_{\mathscr T}$ is the mapping defined on
complex functions $f$ on $V$ by
   \begin{align*}
(\varLambda_{\mathscr T} f) (v) :=
   \begin{cases}
\lambda_v \cdot f\big(\parent v\big) & \text{if } v\in
V^\circ,
   \\
   0 & \text{if } v \text{ is a root of } {\mathscr T}.
   \end{cases}
   \end{align*}
\begin{remark}
If $S_{\lambda}$ is bounded then
$S^*_{\lambda}e_u = {\lambda}_ue_{\mathsf{par}(u)}$ if $u \neq \rootb$, and $S^*_{\lambda}e_u=0$ otherwise.
\end{remark}

Throughout these paper, we will be interested in weighted shifts which are bounded linear. 
The reader is referred to \cite{Jablonski} for the basic theory of weighted shifts on directed trees. In particular, 
it may be concluded from \cite[Proposition 
3.5.1(ii)]{Jablonski} that for a rooted directed tree $\mathscr T$ with root $\mathsf{root}$, the kernel of $S^*_{\lambda}$ is given by
\beq \label{kernel} 
E=\langle e_{\mathsf{root}}\rangle \oplus \bigoplus_{v \in 
V_{\prec}}\left(l^2(\mathsf{Chi}(v)) \ominus \langle \lambdab^v\rangle \right), 
\eeq
where $\lambdab^v : \mathsf{Chi}(v) \rar \mathbb C$ is defined by 
$\lambdab^v(u)=\lambda_u$ and $\langle f \rangle$ denotes the span of 
$\{f\}$.

Recall that a bounded linear operator $T$ on a Hilbert space $\mathcal H$ 
is {\it finitely multicyclic} if there are a finite number of vectors $h_1, \cdots, h_m$
in $\mathcal H$ such that 
\beqn
{\mathcal H} = \bigvee {\{T^k h_1, \cdots, T^k h_m : k \in \mathbb N \}}.
\eeqn
\begin{remark} \label{f-cyclic}
Let $\mathscr T$ be a leafless, locally finite rooted directed tree with
finite branching index and let $S_{\lambda}$ be a bounded weighted shift on $\mathscr T$. It may be concluded from \cite[Proposition 
2.1 and Corollary 2.3]{CT} that $S_{\lambda}$ is finitely cyclic. \end{remark}

\section{Dirichlet Shifts on Directed Trees}

We now introduce the notion of Dirichlet shift on certain directed trees.

\begin{definition}
Let $\mathscr T=(V, \mathcal E)$ be a leafless, locally finite rooted directed tree. 
%\uwam{do not need finite branching index}
For a real number $q \Ge 1,$ consider the weighted shift $S_{\lambda, q}$ on $\mathscr T$ with weights given by
\beq \label{wts} 
\lambda_{u, q} = 
 \frac{1}{\sqrt{\mbox{card}(\child{v})}}\sqrt{\frac{n_v+q}{n_v + 1}}~&\mbox{for~}u \in \child{v},~v \in V,
\eeq
where $n_v$ is the depth of $v \in V$ in $\mathscr T.$
We refer to $S_{\lambda, q}$ as the {\it Dirichlet shift on $\mathscr T$}.
\end{definition}
%{\it All the directed trees appearing in this text are assumed to be  leafless, locally finite, and rooted.} 

Let $S_{\lambda, q}$ be a Dirichlet shift on $\mathscr T$.
Note that the weights of $S_{\lambda, q}$ can be rewritten as \beqn 
\lambda_{v, q} = \sqrt{\frac{n_v+q-1}{n_v}} \frac{1}{\sqrt{\mbox{card}(\sib{v})}}~\mbox{for~}v \in V^{\circ},
\eeqn 
where, for $u \in V$, the {\it siblings} of $u$ is given by \beqn \mathsf{sib}(u) := \begin{cases} \child{\parent{u}} & \mbox{if~}u \neq \mathsf{root}, \\ 
\emptyset & \mbox{otherwise}.
\end{cases} \eeqn
Note further that
\beqn
 \sup_{v \in V}\sum_{u \in \child{v}} \lambda^2_{u, q} = \sup_{v \in V}\frac{n_v+q}{n_v + 1} =q, \quad \inf_{v \in V}\sum_{u \in \child{v}} \lambda^2_{u, q} = \inf_{v \in V}\frac{n_v+q}{n_v + 1} =1.
\eeqn
It now follows from \cite[Propositions 3.1.8 and 3.6.1]{Jablonski} that $S_{\lambda, q}$ is bounded linear, and left-invertible (that is, $S_{\lambda, q}$ is one-one with closed range). 
Thus the Cauchy dual $S'_{\lambda, q}$ of $S_{\lambda, q}$ given by $S_{\lambda, q}(S^*_{\lambda, q}S_{\lambda, q})^{-1}$ is well-defined \cite{Sh}. It turns out that $S'_{\lambda, q}$ is a  weighted shift $S_{\lambda', q}$ on $\mathscr T$ with weights given by
\beq \label{weights-dual}
\lambda'_{v, q} = \sqrt{\frac{n_v}{n_v + q-1}} \frac{1}{\sqrt{\mbox{card}(\sib{v})}} ~ \mbox{for all}~ v \in V^\circ.
\eeq
\begin{remark} \label{classical-D}
In case $\mathscr T$ is a rooted directed tree without any branching vertex (that is, $\mathscr T$ is isomorphic to $\mathbb N$), $S_{\lambda, q}$ is unitarily equivalent to the classical weighted shift $S_{w, q}$ with weights ${w}:=\Big\{\sqrt{\frac{n+q}{n+1}}\Big\}_{n \in \mathbb N}$ ({\it classical Dirichlet shift}). Note that the weights of the Cauchy dual $S'_{w, q}$ of $S_{w, q}$ are $\Big\{\sqrt{\frac{n+1}{n+q}}\Big\}_{n \in \mathbb N}$ ({\it classical Bergman shift}).
\end{remark}

The following classification problem is motivated by \cite[Theorem 9.9]{ACJS}:
\begin{problem} 
%Let $q$ be a positive integer and
%let $\mathscr T_1, \mathscr T_2$ be given locally finite, rooted directed trees of finite branching indices. 
Let $S^{(1)}_{\lambda, q}, S^{(2)}_{\lambda, q}$ be Dirichlet shifts on $\mathscr T_1, \mathscr T_2$ respectively. Find sufficient and necessary conditions on $\mathscr T_1$ and $\mathscr T_2$ which ensure that $S^{(1)}_{\lambda, q}$ and $S^{(2)}_{\lambda, q}$ are unitarily equivalent.
\end{problem}

Note that the above problem has the following solution in case $q=1.$ 
{\it The shifts $S^{(1)}_{\lambda, 1}$ and $S^{(2)}_{\lambda, 1}$ unitarily equivalent if and only if $$\sum_{v \in V^{(1)}_{\prec}}\Big(\mbox{card}(\child{v})-1\Big)=\sum_{v \in V^{(2)}_{\prec} }\Big(\mbox{card}(\child{v})-1\Big),$$
where $V^{(1)}_{\prec}$ and $V^{(2)}_{\prec}$ are set of branching vertices of $\mathscr T_1$ and $\mathscr T_2$ respectively.}
Indeed, since $S^{(1)}_{\lambda, 1}$ and $S^{(2)}_{\lambda, 1}$ are  isometries (Proposition \ref{p-iso} below),
by von Neumann-Wold decomposition \cite{Co},  $S^{(1)}_{\lambda, 1}$ and $S^{(2)}_{\lambda, 1}$ unitarily equivalent if and only if $\dim \ker (S^{(1)}_{\lambda, 2})^*=\dim \ker (S^{(2)}_{\lambda, 2})^*$. The desired conclusion is now immediate from \eqref{kernel}. This has been recorded in more generality in \cite[Theorem 9.9(ii)]{ACJS}.

One of the main results of this note provides a solution to the above problem in case $q$ is a positive integer.
\begin{theorem} \label{solution}
Let $q$ be an integer bigger than $1.$
For $j=1, 2$, let
$S^{(j)}_{\lambda, q}$ be the Dirichlet shift on rooted directed tree $\mathscr T_j= (V_j, \mathcal E_j)$ with $\rootb_j$, let
$\mathcal G^{(j)}_n:=\{v \in V_j : v \in \childn{n}{\rootb_j}\}~(n \in \mathbb N),$ and $E_j=\ker (S^{(j)}_{\lambda, q})^*$. Then 
$S^{(1)}_{\lambda, q}$ is unitarily equivalent to $S^{(2)}_{\lambda, q}$ if and only if
for every $n \in \mathbb N,$ 
%$\mbox{card}(V^{(1)}_{\prec} \cap \mathcal G^{(1)}_n) =\mbox{card}(V^{(2)}_{\prec} \cap \mathcal G^{(2)}_n)$ and 
\beq \label{constant-g} \sum_{v \in V^{(1)}_{\prec} \cap \mathcal G^{(1)}_n}\Big(\mbox{card}(\child{v})-1\Big)=\sum_{v \in V^{(1)}_{\prec} \cap \mathcal G^{(2)}_n}\Big(\mbox{card}(\child{v})-1\Big).\eeq
\end{theorem}
\begin{remark}
Since $\mathscr T_j$ is locally finite,  the $V^{(j)}_{\prec} \cap \mathcal G^{(j)}_n$ is finite for every $n \in \mathbb N$ and $j=1, 2$, and hence the sums appearing in \eqref{constant-g} are finite.
Further, it may happen that \eqref{constant-g} holds for two non-isomorphic directed trees (see \cite[Figure 2]{ACJS}). 
\end{remark}

%\begin{figure}
%\includegraphics[scale=.5]{non-iso.pdf} \caption{Non-isomorphic trees satisfying the condition \eqref{constant-g}}
%\end{figure}

%\begin{proposition} \label{solution}
%For $j=1, 2$, let
%$S^{(j)}_{\lambda, 2}$ be the Dirichlet shift on $\mathscr T_j= (V_j, \mathcal E_j)$, let
%$\mathcal G^{(j)}_n:=\{v \in V_j : v \in \childn{n}{\rootb_j}\}~(n \in \mathbb N),$ and $E_j=\ker (S^{(j)}_{\lambda, 2})^*$. Then 
%$S^{(1)}_{\lambda, 2}$ is unitarily equivalent to $S^{(2)}_{\lambda, 2}$ if and only if
%for every $n \in \mathbb N,$ 
%%$\mbox{card}(V^{(1)}_{\prec} \cap \mathcal G^{(1)}_n) =\mbox{card}(V^{(2)}_{\prec} \cap \mathcal G^{(2)}_n)$ and 
%\beq \label{constant-g} \sum_{v \in V^{(1)}_{\prec} \cap \mathcal G^{(1)}_n}\Big(\mbox{card}(\child{v})-1\Big)=\sum_{v \in V^{(1)}_{\prec} \cap \mathcal G^{(2)}_n}\Big(\mbox{card}(\child{v})-1\Big).\eeq
%\end{proposition}

The case $q=2$ of Theorem \ref{solution} is a special case of \cite[Theorem 9.9(i)]{ACJS}. In what follows, we provide an alternative verification of this fact based on modeling $S_{\lambda, 2}$ as a multiplication by $z$ on a vector-valued Dirichlet space. With this identification, the problem essentially reduces to classification problem of multiplication operators on vector-valued Dirichlet spaces (refer to Section 3).
This part of the proof relies on the theory of vector-valued Dirichlet spaces as expounded in \cite{O}. The rather involved proof of the general case, as presented in the last section, relies on tree analogs of weighted Bergman spaces. These spaces can be seen as Cauchy dual of Dirichlet spaces in the sense of S. Shimorin \cite{Sh}. 

%\subsection{Structural Properties of Dirichlet Shifts}

In the remaining part of this section, we derive some structural properties of the weighted shifts $S_{\lambda, q}$ on $\mathscr T$. 
Before we state formulae for moments of $S_{\lambda, q}$ and $S_{\lambda', q},$ recall that the {\it Pochhammer symbol} is defined by 
\beqn
(x)_y=\frac{\Gamma(x+y)}{\Gamma(x)},
\eeqn
%whenever $x$ and $x+y$ are off the set of negative integers, 
where $\Gamma$ is the gamma function defined for all complex numbers except the non-positive integers.
\begin{lemma} \label{formula-moments}
%Let $\mathscr T=(V, \mathcal E)$ be a locally finite, rooted directed tree with finite branching index.
%For a real number $q \Ge 1,$ consider the weighted shift $S_{\lambda, q}$ on $\mathscr T$ with weights given by \eqref{wts} 
Let $S_{\lambda, q}$ be a Dirichlet shift on $\mathscr T=(V, \mathcal E)$
and let $S_{\lambda', q}$ be the Cauchy dual of $S_{\lambda, q}.$ Then for $k \in \mathbb N$ and $v \in V,$
\beqn
\|S^k_{\lambda, q} e_v\|^2=\frac{(n_v +q)_k}{(n_v+1)_k}, \quad \|S^k_{\lambda', q} e_v\|^2=\frac{(n_v +1)_k}{(n_v+q)_k}.
\eeqn
In particular, the spectral radii of $S_{\lambda, q}$ and $S_{\lambda', q}$ are $1.$
\end{lemma}
\begin{proof}
We verify the first formula by induction on integers $k \Ge 0$ for a fixed $v \in V$. The formula is trivial for $k=0.$ Suppose the formula holds for some integer $k \Ge 0.$ 
Since $\childn{n}{u}$ and $\childn{n}{w}$ are disjoint for distinct vertices $u$ and $w,$ it follows from \cite[Lemma 6.1.1(i)]{Jablonski} that
$\{S^k_{\lambda, q}e_w\}_{w \in \child{v}}$ is mutually orthogonal. Also, since $n_w=n_v+1$ for $w \in \child{v}$, we obtain
\beqn
\|S^{k+1}_{\lambda, q} e_v\|^2 &=& \Big\|S^k_{\lambda, q} \sum_{w \in \child{v}} \lambda_{w, q}e_w \Big\|^2  = \sum_{w \in \child{v}} \lambda^2_{w, q} \|S^k_{\lambda, q} e_w \|^2 \\ &\overset{\eqref{wts}}=& \sum_{w \in \child{v}}\frac{1}{{\mbox{card}(\child{v})}}{\frac{n_v+q}{n_v + 1}}\,\frac{(n_w +q)_k}{(n_w+1)_k} \\
&=& {\frac{n_v+q}{n_v + 1}}\,\frac{(n_v +q+1)_k}{(n_v+2)_k} 
= \frac{(n_v +q)_{k+1}}{(n_v+1)_{k+1}}.
\eeqn
The second formula can be verified similarly. To see the remaining part, note that 
\beqn
\|S^k_{\lambda, q}\|=\sup_{n \in \mathbb N}\sqrt{\frac{(n +q)_k}{(n+1)_k}}, \quad \|S^k_{\lambda', q}\|=\sup_{n \in \mathbb N} \sqrt{\frac{(n +1)_k}{(n+q)_k}},
\eeqn
and apply the spectral radius formula.
\end{proof}

%In this section, we discuss some basic structural properties of the multiplication operator $\mathscr M_{z, q}$ acting on the $\mathscr T$-analog $\mathscr H_q$ of the classical Dirichlet space $D_q.$

The second part of the following generalizes \cite[Proposition 8]{At}, \cite[Theorem 8.6]{AK}.
\begin{proposition} \label{p-iso}
%Let $\mathscr T = (V, \mathcal E)$ be a locally finite, rooted directed tree of finite branching index and consider the weighted shift $S_{\lambda, q}$ on $\mathscr T$ with weights given by \eqref{wts}.
Let $S_{\lambda, q}$ be a Dirichlet shift on $\mathscr T=(V, \mathcal E)$
and let $S_{\lambda', q}$ be the Cauchy dual of $S_{\lambda, q}.$
If $q$ is a positive integer, then we have the following: 
\begin{itemize}
\item [(i)] $S_{\lambda', q}$ is subnormal, that is, $S_{\lambda', q}$ admits a normal extension.
\item [(ii)] $S_{\lambda, q}$ is a $q$-isometry, that is,
\beqn
\sum_{k=0}^q (-1)^k {q \choose k} S^{*k}_{\lambda, q} S^k_{\lambda, q}=0,
\eeqn
but not a $(q-1)$-isometry,
\item[(iii)] if $\mathscr T$ is of finite branching index, then the self-commutator $[S^*_{\lambda, q}, S_{\lambda, q}]:=S^*_{\lambda, q}S_{\lambda, q}-S_{\lambda, q}S^*_{\lambda, q}$ of $S_{\lambda, q}$ is of trace-class,
\item[(iv)] $S_{\lambda, 2}$ has wandering subspace property, that is, for any $S_{\lambda, 2}$-invariant subspace $\mathcal M$ of $l^2(V)$,
\beqn
\mathcal M = \bigvee_{k \in \mathbb N}\{S^k_{\lambda, 2}f : f \in \mathcal M \ominus S_{\lambda, 2} \mathcal M\}.
\eeqn
\end{itemize}
\end{proposition}
\begin{proof}
Suppose that $q$ is a positive integer.

(i) By \cite[Theorem 6.1.3]{Jablonski}, $S_{\lambda', q}$ is subnormal if and only if for every $v \in V,$ $\{\|S^k_{\lambda', q}e_v\|^2\}_{k \in \mathbb N}$ is a Hausdorff moment sequence. However, by Lemma \ref{formula-moments}, 
for $v \in V$ and $k \in \mathbb N.$ 
 \begin{align*}
\|S^k_{\lambda', q} e_v\|^2=\frac{(n_v +1)_k}{(n_v+q)_k}=
   \begin{cases}
1  & \text{if } q=1,
   \\
   \frac{(n_v+1) \cdots (n_v+q-1)}{(n_v+k+1)\cdots (n_v+k+q-1)} & \text{if } ~q \Ge 2.
   \end{cases}
   \end{align*}
Since $\{\frac{1}{k+l}\}_{ k \in \mathbb N}$ is a Hausdorff moment sequence for any integer $l \Ge 1,$ by general theory \cite{BCR}, so is $\{\|S^k_{\lambda', q}e_v\|^2\}_{k \in \mathbb N}$. 

(ii) 
Since the sequence $\{S^k_{\lambda, q}e_v\}_{v \in V}$ is orthogonal, it is sufficient to check that \beqn
\sum_{k=0}^q (-1)^k {q \choose k} \|S^k_{\lambda, q}e_v\|^2=0~\mbox{for every}~v \in V.
\eeqn
However, by Lemma \ref{formula-moments},
for $v \in V$ and $k \in \mathbb N,$ 
 \begin{align*}
\|S^k_{\lambda, q} e_v\|^2=\frac{(n_v +q)_k}{(n_v+1)_k}=
   \begin{cases}
1  & \text{if } q=1,
   \\
   \frac{(n_v+k+1)\cdots (n_v+k+q-1)}{(n_v+1) \cdots (n_v+q-1)} & \text{if } ~q \Ge 2.
   \end{cases}
   \end{align*}
In any case, the sequence  $\{\|S^k_{\lambda, q}e_v\|^2\}_{k \in \mathbb N}$ is a polynomial in $k$ of degree $q-1.$ By \cite[Proof of Lemma 2.5]{CK}, $S_{\lambda, q}$ is a $q$-isometry, but not a $(q-1)$-isometry.

(iii) Assume that $\mathscr T$ is of finite branching index. By Remark \ref{f-cyclic}, $S_{\lambda', q}$ is finitely cyclic.  By part (ii) above,  $S_{\lambda', q}$ is subnormal. Hence, by Berger-Shaw Theorem \cite{Co}, $S_{\lambda', q}$ has trace-class self-commutator.
Since $A:=S_{\lambda, q}$ is right Fredholm with right essential inverse $(A^*A)^{-1}A^*$, the desired conclusion may now be derived from the following identity: $$[A^*,A]A = -A^*A([A'^*,A']A)A^*A,$$ where $A':=A(A^*A)^{-1}$.

(iv) This is immediate from part (ii), \cite[Lemma 3.3]{CT} and \cite[Theorem 1]{R}.
\end{proof}
%\begin{remark}
%Note that $S_{\lambda, q}$ is not a $(q-1)$-isometry.
%\end{remark}
\begin{remark}
Assume that $\mathscr T$ is of finite branching index.
Since essential spectral picture of a finitely multicyclic, completely non-unitary $q$-isometry $T$ coincides with the unilateral shift of multiplicity $\dim \ker T^*$ \cite{Ag-St}, it may be concluded from the BDF Theorem \cite{BDF} that
$S_{\lambda, q}$ is unitarily equivalent to a compact perturbation of the unilateral shift of multiplicity $\dim \ker S^*_{\lambda, q}.$ 
\end{remark}

%Here is one immediate consequence of the preceding result. Applying the previous result to $\mathscr M=\mathscr H_2$, we may conclude that $\mathscr M_{z, 2}$ is finitely cyclic.
%Hence, by \cite[Proposition 2.21]{C}, the self-commutator $[\mathscr M^*_{z, 2}, \mathscr M_{z, 2}]:=\mathscr M^*_{z, 2}\mathscr M_{z, 2}-\mathscr M_{z, 2}\mathscr M^*_{z, 2}$ of $\mathscr M_{z, 2}$ is trace-class.

%\begin{corollary}
%Let $\mathscr T = (V, \mathcal E)$ be a locally finite, rooted directed tree of finite branching index and consider the weighted shift $S_{\lambda, q}$ on $\mathscr T$.
%Consider the $\mathscr T$-analog $\mathscr H_q$ of the classical Dirichlet space.
%Let $\mathscr M_{z, q}$ denote the operator of multiplication by $z$ on $\mathscr H_q.$ If $q$ is a positive integer, then the self-commutator $[\mathscr M^*_{z, q}, \mathscr M_{z, q}]:=\mathscr M^*_{z, q}\mathscr M_{z, q}-\mathscr M_{z, q}\mathscr M^*_{z, q}$ of $\mathscr M_{z, q}$ is trace-class.
%\end{corollary}
%\begin{proof} Suppose that $q$ is a positive integer. 
%\end{proof}

%\begin{corollary}
%Let $\mathscr M_{z, q}$ denote the operator of multiplication by $z$ on $\mathscr H_q.$ If $q$ is a positive integer, then the self-commutator $[\mathscr M^*_{z, q}, \mathscr M_{z, q}]:=\mathscr M^*_{z, q}\mathscr M_{z, q}-\mathscr M_{z, q}\mathscr M^*_{z, q}$ of $\mathscr M_{z, q}$
%is compact. In particular, $[\mathscr M^*_{z, 2}, \mathscr M_{z, 2}]$ is trace-class.
%\end{corollary}
%\begin{proof}
%Since 
%\end{proof}

The shift operators $S_{\lambda, q}~(q \Ge 1)$ provide {\it new} examples of finitely multicyclic $q$-isometries in the following sense (cf. \cite[Remark 4.5]{J}).
\begin{proposition}
%Let $\mathscr T = (V, \mathcal E)$ be a locally finite, rooted directed tree of finite branching index and 
%let $S_{\lambda, q}$ be the weighted shift on $\mathscr T$ with weights given by \eqref{wts}.
Let $S_{\lambda, q}$ be a Dirichlet shift on a directed tree $\mathscr T=(V, \mathcal E)$ of finite branching index and
let $S_{w, q}$ be the classical Dirichlet shift (see Remark \ref{classical-D}). Then $S_{\lambda, q}$ is unitarily equivalent to any finite orthogonal sum of  $S_{w, q}$ if and only if either $q=1$ or $\mathscr T$ is isomorphic to $\mathbb N.$ 
\end{proposition}
\begin{proof}  Note that $l=\dim \ker S^*_{\lambda, q}$ is finite by \cite[Proposition 2.1]{CT}.
If $\mathscr T$ is isomorphic to $\mathbb N$ then clearly $S_{\lambda, q}$ is unitarily equivalent to $S_{w, q}$. 
Further, if $q=1$ then by the von Neumann-Wold decomposition for isometries \cite{Co}, $S_{\lambda, q}$ is unitarily equivalent to orthogonal sum of $\dim \ker S^*_{\lambda, q}$ copies of $S_{w, q}$. 
This gives the sufficiency part.
To see the necessity part, suppose that
$S_{\lambda, q}$ is unitarily equivalent to orthogonal sum $S^{(l)}_{w, q}$ of $l$ copies of $S_{w, q}$.  Note that the Cauchy dual $S_{\lambda', q}$ of $S_{\lambda, q}$ is unitarily equivalent to $(S'_{w, q})^{(l)}$. By \cite[Example 2.7]{CK}, the defect operator $\mathcal D_{S'_{w, q}}$ is an orthogonal projection onto $\ker S'^{*}_{w, q}$, where, for a bounded linear operator $T$, $\mathcal D_T$ is given by
\beqn
\mathcal D_{T}:=\sum_{k=0}^{q} (-1)^{k} {q \choose k} T^{k} T^{*k}
\eeqn  
%This has been observed in \cite[Pg 618]{He} in case $q=2.$ The general case can be obtained by explicit computations.
The essential part of the proof shows that the defect operator $\mathcal D_{S_{\lambda', q}}$ is never an orthogonal projection unless $q=1$ or $\mathscr T$ is isomorphic to $\mathbb N.$ 

We may assume that $\mathscr T$ is not isomorphic to $\mathbb N.$ Thus 
 $V_\prec$ is nonempty. Further, since $\mathscr T$ is locally finite with finite branching index, we can choose $v \in V_\prec$ such that $\mathsf{Chi}(v) = \{u_1,u_2, \cdots, u_m \}$ and $\mbox{card}(\mathsf{Chi}(u_j))=1~(j=1, \cdots, m)$ for some positive integer $m \Ge 2$. Let $\mathsf{Chi}(u_j)= \{w_j\}~(j=1, \cdots, m)$ and $f_v = \displaystyle \sum_{j=1}^{m} f(w_{j}) e_{w_{j}}$ be such that $$\sum_{j=1}^{m} f(w_{j}) = 0,~\mbox{and~} f(w_{j}) \neq 0~\mbox{for some~} j=1, \cdots, m.$$ Note that $S^*_{\lambda', q} (f_v) \neq 0$ and $S^{*k}_{\lambda', q} (f_v) =0$ for all integers $k > 1$.  It then follows that
\begin{eqnarray*}
\sum_{k=0}^{q} (-1)^{k} {q \choose k} S_{\lambda', q}^{k} S_{\lambda', q}^{*k}(f_v) & = &  f_v - q \sum_{j=1}^m f(w_{j})\lambda'_{w_{j}} e_{w_{j}} \\
& \overset{\eqref{weights-dual}}= & f_v - q \sum_{j=1}^n f(w_{j})\frac{n_{w_{j} }}{n_{w_{j}}+q-1} \, e_{w_{j} } \\
& = & \left( 1- q\big( \frac{n_{v}+2} {n_{v}+q+1} \big) \right)f_v. 
\end{eqnarray*}
It is easy to see from \eqref{kernel} that $f_v$ is orthogonal to $ \ker(S^*_{\lambda', q})$. Thus if $D_{S_{\lambda', q}}$ is the orthogonal projection onto $\ker(S^*_{\lambda', q})$ then we must have $$\left( 1- q\big( \frac{n_{v}+2} {n_{v}+q+1} \big) \right)f_v=0,$$ that is,
 $ (1 -q) (n_{v}+1) = 0 .$ This is possible only if $q=1.$
\end{proof}

\section{Dirichlet Spaces Associated with Directed Trees}

The following result enables us to associate a Dirichlet space with every leafless, locally finite rooted directed tree. It is worth mentioning 
that certain Hardy-type spaces are associated with some infinite acyclic, undirected, connected graphs in \cite{AV} (refer also to \cite[Section 4.3]{ARSW} for a version of Dirichlet space on the Bergman tree).
\begin{proposition} \label{S-c-a-kernel-dim1}
%Let $\mathscr T = (V, \mathcal E)$ be a locally finite, rooted %directed tree of finite branching index. Let $S_{\lambda, q}$ be %the weighted shift on $\mathscr T$ with weights given by\eqref{wts} 
Let $S_{\lambda, q}$ be a Dirichlet shift on $\mathscr T=(V, \mathcal E)$ and let $E:=\ker S^*_{\lambda, q}$. 
Then there exist a $z$-invariant reproducing kernel Hilbert space 
$\mathscr H_q$ of $E$-valued holomorphic functions defined on the disc
$\mathbb{D}$ and a unitary mapping $U:l^2(V) \longrightarrow\mathscr H_q$ such
that ${\mathscr M}_{z, q}U=US_{\lambda, q},$ where 
${\mathscr M}_{z, q}$ 
denotes the operator of multiplication by $z$ on $\mathscr H_q$.
%Then $S_{\lambdab, q}$ is unitarily equivalent to the multiplication operator $\mathscr M_{z, q}$ on a reproducing kernel 
%Hilbert space $\mathscr H_q$ of $E$-valued holomorphic functions on unit disc $\mathbb D$, where 
Further, $U$ maps $E$ onto the subspace $\mathscr E$
of $E$-valued constant functions in $\mathscr H_q$ such that $Ug=g$ for every $g \in E.$ 
%If, in addition, $\mathscr T$ is of finite branching index, then
Furthermore, we have the following:
\begin{itemize}
\item [(i)] the reproducing kernel $\kappa_{\mathscr H_q} : \mathbb D 
\times \mathbb D \rar B(E)$ associated with 
$\mathscr H_q$ satisfies $\kappa_{\mathscr H_q}(\cdot,w)g \in \mathscr H_q$ and
$
\inp{Uf}{\kappa_{\mathscr H_q}(\cdot,w)g}_{\mathscr H_q} = \inp{(Uf)(w)}{g}_E$
for every $f \in l^2(V)$ and $g \in E,$
\item [(ii)] $\kappa_{\mathscr H_q}$ is given by
 \beqn
\kappa_{\mathscr H_q}(z, w) &=& \sum_{n=0}^{\infty}\frac{(1)_n}{(q)_n}~ {z^n \overline{w}^n} ~P_{\langle e_\rootb \rangle} \\  &+& 
\sum_{v \in V_{\prec}} \sum_{n=0}^{\infty}  \frac{(n_v +2)_n}{(n_v + q+1)_n} {z^n \overline{w}^n}~P_{l^2(\child{v}) \ominus \langle \lambdab^v \rangle}~(z, w \in \mathbb D),
\eeqn
where $P_{\mathscr M}$ denotes the orthogonal projection of $\mathscr H$ onto the subspace $\mathscr M$ of $\mathscr H$,
 \item [(iii)] The $E$-valued polynomials in $z$ are dense in $\mathscr H_q$. In fact,
$$\mathscr H_q=\bigvee\{z^nf:f\in {\mathscr E},\ n \in \mathbb N\}.$$
\item [(iv)] If $\mathscr B$ is an orthonormal basis of $\mathscr E$ then $\{z^nf : f \in \mathscr B, ~n \in \mathbb N\}$ forms an orthogonal basis of $\mathscr H_q.$
 \end{itemize}
\end{proposition}
\begin{remark}
Note that the reproducing kernel Hilbert space $\mathscr H_1$ is nothing but the vector-valued Hardy space associated with the kernel $\frac{I_E}{1-z\overline{w}}~(z, w \in \mathbb D),$ where $I_E$ denotes the identity operator on $E.$ In view of the decomposition \eqref{kernel} of $E,$
this is immediate from the result above.
\end{remark}

\begin{proof}
%The proof relies on a model theorem of Shimorin for left-invertible operators \cite{Sh}. 
The proof is an application of \cite[Theorem 2.2]{CT}.
%For sake of convenience, we outline the proof.
The first half and parts (i), (iii) follow from \cite[Theorem 2.2]{CT}. 
Thus we only need to verify parts (ii) and (iv).
%Assume that $\mathscr T$ is of finite branching index.
Note first that by \cite[Theorem 2.2]{CT}(ii) and Lemma \ref{formula-moments}, $\kappa_{\mathscr H_q}$ is given by
 \begin{equation} \label{eq2}
\kappa_{\mathscr H_q}(z,w)=I_E+\displaystyle\sum_{{j,k \Ge 1}}C_{j,k}z^j\overline{w}^k~(z, w \in \mathbb D),
 \end{equation}
 where $I_E$ denotes the identity operator on $E$, and $C_{j,k}$ are bounded 
linear operators on $E$ given by 
 $$C_{j,k}=P_ES^{*j}_{\lambda', q}S_{\lambda', q}^k|_E~(j, k =1, 2, \cdots)$$ with 
$P_E$ 
being the orthogonal projection of $l^2(V)$ onto $E$. 
To see that $C_{j, k}=0$ for $j \neq k$, we need the following identity \cite[(5.6)]{CPT}:
\beq \label{card-ide}
\sum_{u \in \childn{k}{v}} \prod_{l=0}^{k-1} \frac{1}{s_{l, u}} = 1~\mbox{for~}v \in V~\mbox{and~}k \Ge 1,
\eeq
where, for a positive integer $l$ and $v \in V,$ $s_{l, v}:=\mbox{card} (\sib{\parentn{l}{v}})$.
This identity for a fixed $v \in V$ can be verified by induction on integers $k \Ge 1$. 
Now for $v \in V$ and $j,k \Ge 1$, 
\beqn \label{power-S-lambda}
S^k_{\lambda', q} e_v = \displaystyle \sqrt{\frac{(n_v +1)_k}{(n_v+q)_k}} \sum_{u \in \childn{k}{v}}  \prod_{l=0}^{k-1} \frac{1}{\sqrt{s_{l, u}}} e_u,
\eeqn
%S^{*j}_{\lambda', q} e_v  &=& \sqrt{\frac{(n_v +q- 1)_{-j} }{(n_v)_{-j}}} ~\prod_{i=0}^{j-1} \frac{1}{\sqrt{\mbox{card} (\sib{\parentn{i}{v}})}} e_{\parentn{j}{v}}, \quad
%\eeqn
\beqn 
S^{*j}_{\lambda', q} e_v= 
\begin{cases} \displaystyle \sqrt{\frac{(n_v +q)_{-j} }{(n_v+1)_{-j}}} ~\prod_{l=0}^{j-1} \frac{1}{\sqrt{s_{l, v}}} e_{\parentn{j}{v}}~&\mbox{if~}n_v \Ge j, \\
0~& \mbox{otherwise}.
\end{cases}  
\eeqn
% Also, for positive integers $j, k$ and $v \in V$ such that $\parentn{j-k}{v}$ is nonempty, set
%\beqn
%\beta_{j, k}(v, q) &:=& \prod_{i=0}^{j-k-1} \frac{1}{\sqrt{s_{i, v}}} \sqrt{\frac{(n_v +q -1)! (n_v + k)!}{(n_v + k + q -1)! n_v !}} \\ &\times& \sqrt{\frac{(n_v +k +q - j -1)! (n_v +k) !}{(n_v +k +q-1)!(n_v +k -j)!}}.
%\eeqn
%Note that for every $v \in V,$ $\beta_{j, k}(\cdot, a)$ is constant on $\child{v}.$
Let $v \in V$ and $j \Ge k$. Note that if $\parentn{j-k}{v}$ is empty, then
$S^{*j}_{\lambda', q} S^k_{\lambda', q} e_v =0.$ Otherwise
\beqn
S^{*j}_{\lambda', q} S^k_{\lambda', q} e_v &=& \sqrt{\frac{(n_v +1)_k}{(n_v+q)_k}} \sum_{u \in \childn{k}{v}}  \prod_{l=0}^{k-1} \frac{1}{\sqrt{s_{l, u}}} S^{*j}_{\lambda', q}e_u\\
&=& \sqrt{\frac{(n_v +1)_k}{(n_v+q)_k}}  \sum_{u \in \childn{k}{v}} \Big( \prod_{l=0}^{k-1} \frac{1}{\sqrt{s_{l, u}}}\\ & \times &  \sqrt{\frac{(n_u +q)_{-j} }{(n_u+1)_{-j}}}  ~\prod_{m=0}^{j-1} \frac{1}{\sqrt{s_{m, u}}} \Big)e_{\parentn{j}{u}}\\
%&=& \beta_{j, k}(v, q) \sum_{u \in \childn{k}{v}} \prod_{l=0}^{k-1} \frac{1}{s_{l, u}} e_{\parentn{j-k}{v}}\\
&=& \sqrt{\frac{(n_v +1)_k}{(n_v+q)_k}}~\sqrt{\frac{(n_v+k +q)_{-j} }{(n_v+k+1)_{-j}}} ~\prod_{l=0}^{j-k-1} \frac{1}{\sqrt{s_{l, v}}}  e_{\parentn{j-k}{v}},
\eeqn
where the last equality follows from \eqref{card-ide}, $n_u=n_v+k$ and $s_{l, v}=s_{l+k, u}.$ This also shows the following: 
\begin{enumerate}
\item[(a)] For $k \in \mathbb N,$ \beq \label{moment} S^{*k}_{\lambda', q} S^k_{\lambda', q} e_v=\frac{(n_v +1)_k}{(n_v+q)_k}\,e_v. \eeq  Here we used the convention that product over the empty set equals $1.$
\item[(b)] For non-negative integers $j > k$, \beq \label{moment-mixed} S^{*j}_{\lambda', q} S^k_{\lambda', q} e_v=\beta(j, k, v) e_{\parentn{j-k}{v}}\eeq for some positive scalar $\beta(j, k, v)$ such that 
\beq \label{beta-con}
\beta(j, k, \omega) =\beta(j, k, v)~\mbox{ for all}~ \omega \in \sib{v}.
\eeq
\end{enumerate}

Let 
$f \in E$. 
Note that by \eqref{kernel}, $f$ takes the form
$f = f_{\rootb} + \sum_{w \in  V_{\prec}} f_w $, 
where $f_{\rootb} = \gamma e_{\mathsf{root}}$ for some $\gamma \in \mathbb C$ and $f_w = \sum_{v \in \child{w}} f(v)e_v$ such that $$\sum_{v \in \child{w}} f(v) \lambda_v = 0~\mbox{for~} w \in V_\prec.$$
In view of \eqref{wts}, $\lambda_v$ is constant on $\child{w}$, and hence we obtain that \beq\label{sum-0}  \sum_{v \in \child{w}} f(v)  = 0~\mbox{for all~} w \in V_\prec. \eeq 
%Also, note that $\sum_{w \in \child{v}} f(w)\lambda_w = 0$ for all $v \in V_\prec$. Since $\lambda_w$ is constant on siblings, it follows that $\sum_{w \in \child{v}} f(w) = 0$. 
It follows that for $w \in V_\prec$ and $j >k$, 
\beq \label{orthog}
S^{*j}_{\lambda', q} S^k_{\lambda', q}\,f_w= \sum_{v \in \child{w}} f(v) S^{*j}_{\lambda', q} S^k_{\lambda', q}\, e_v  
= 0,
\eeq
where we used \eqref{moment-mixed}, \eqref{beta-con} and \eqref{sum-0}.
%Almost the same calculations show that 
%\beqn \label{beta-k-v-a}
%S^{*k}_{\lambda', q} S^k_{\lambda', q} f_v =
%\frac{(n_v +q)! (n_v + k+1)!}{(n_v + k + q)! (n_v +1)!}~f_v \quad \mbox{for every~}k \in \mathbb N.
%\eeqn
It may now be concluded from \eqref{eq2}, \eqref{kernel}, \eqref{moment}  that the reproducing kernel $\kappa_{\mathscr H_q}$ takes the required form. Note that the conclusion in \eqref{orthog} 
also holds for $S_{\lambda, q}$ (by the same reasoning), and hence 
the sequence $\{z^n \mathscr E: n \in \mathbb N\}$ of subspaces of $\mathscr H_q$ is mutually orthogonal. This combined with part (iii) yields (iv).
%Also, since $E$ is finite dimensional by \cite[Proposition 2.1]{CT}, we note that for every $w \in \mathbb D$, $\kappa_{\mathscr H_a}(\cdot, w)$ is a sum of finitely many power series in $z$ each converging on the unit disc $\mathbb D$.
%We leave the routine verification of (i) to the interested reader.
\end{proof}
\begin{remark}
Note that $\kappa_{\mathscr H_2}$ takes the form
 \beqn
\kappa_{\mathscr H_2}(z, w) &=& -\frac{1}{z\overline{w}}\log(1 - z\overline{w}) ~P_{\langle e_\rootb \rangle} \\  &+& 
\sum_{v \in V_{\prec}} \sum_{n=0}^{\infty}  \frac{n_v +2}{n_v + 2+n} {z^n \overline{w}^n}~P_{l^2(\child{v}) \ominus \langle \lambdab^v \rangle} \\
&=& -\frac{1}{z\overline{w}}\log(1 - z\overline{w}) ~P_{\langle e_\rootb \rangle} \\  &-& 
\sum_{v \in V_{\prec}} \frac{n_v +2}{z^{n_v+1}\overline{w}^{n_v+1}} \left(\frac{1}{z\overline{w}}\log(1 - z\overline{w})+\sum_{k=0}^{n_v} \frac{z^k\overline{w}^k}{k+1}\right)~P_{l^2(\child{v}) \ominus \langle\lambdab^v \rangle}
\eeqn
for $z, w \in \mathbb D \setminus \{0\}.$ A particular case in which $\mathscr T$ has only branching point at $\rootb$, $\kappa_{\mathscr H_2}$ simplifies to
\beqn
\kappa_{\mathscr H_2}(z, w) &=& -\frac{1}{z\overline{w}}\log(1 - z\overline{w}) ~P_{\langle e_\rootb \rangle} \\ &-&
\frac{2}{z\overline{w}} \left(\frac{1}{z\overline{w}}\log(1 - z\overline{w})+1\right)~P_{l^2(\child{\rootb}) \ominus \langle \lambdab^{\rootb} \rangle}
\eeqn
for $z, w \in \mathbb D \setminus \{0\}.$
\end{remark}

%\begin{remark}
Note that in case $\mathscr T$ is a rooted directed tree without any branching vertex (that is, $\mathscr T$ is isomorphic to $\mathbb N$), $\mathscr H_q$ is nothing but the {\it classical Dirichlet space} $D_q~(q \Ge 1)$, that is, the reproducing kernel Hilbert space associated with reproducing kernel $$\sum_{n=0}^{\infty}\frac{(1)_n}{(q)_n}~ {z^n \overline{w}^n}~(z, w \in \mathbb D).$$
%Similar remark holds for the spaces $\mathscr H_{\textendash q}~(q \Ge 1)$
(refer to \cite{EKMR} for the basic theory of classical Dirichlet spaces; the reader is also referred to \cite{ARSW} for an interesting exposition on some recent developments related to Dirichlet spaces).  This motivates the following definition.
%\end{remark}

%\begin{remark} \label{rem-D-B}
%It is evident from the proof of Proposition \ref{S-c-a-kernel-dim1} that one can associate a functional Hilbert space, say, $\mathscr H_{\textendash q}$ with the Cauchy dual $S_{\lambda', q}$ of the Dirichlet shift $S_{\lambda, q}.$ It is also clear that $S_{\lambda', q}$ is unitarily equivalent to the operator of multiplication by $z$ on $\mathscr H_{\textendash q}$ (see \cite[Proposition 5.1]{CPT}).
%\end{remark}

\begin{definition} 
Let $\mathscr T = (V, \mathcal E)$ be a leafless, locally finite rooted directed tree. 
%of finite branching index. 
We refer to the space $\mathscr H_q$, as constructed in Proposition \ref{S-c-a-kernel-dim1}, as the {\it Dirichlet space associated with $\mathscr T$.}
%Further, we refer to the space $\mathscr H_{\textendash q}$ as the {\it Bergman space associated with $\mathscr T$.}
\end{definition}

%We refer to the space $\mathscr H_{q}$ as the {\it $\mathscr T$-analog of the classical Dirichlet space $D_q$}. 
%The space $\mathscr H_{q}$ provides an example of non-commutative reproducing kernel Hilbert space (refer to \cite{B} for a systematic study of such spaces). \uwam{verify}

%\begin{proposition} 
%Consider the multiplication operator $\mathscr M_{z, q}$ on the $\mathscr T$-analog of the Dirichlet space $\mathscr H_q$ and let $E=\ker \mathscr M^*_{z, q}.$ 
%\end{proposition}

\begin{corollary} \label{D-norm}
%Let $\mathscr T = (V, \mathcal E)$ be a locally finite, rooted directed tree of finite branching index and consider the weighted shift $S_{\lambda, q}$ on $\mathscr T$.
%Consider the $\mathscr T$-analog $\mathscr H_q$ of the classical Dirichlet space. 
Let $\mathscr H_q$ be a Dirichlet space associated with $\mathscr T=(V, \mathcal E)$ and let $\mathscr M_{z, q}$ be the operator of multiplication by $z$ on $\mathscr H_q.$
Then, for any $f(z)=\sum_{n=0}^{\infty}f_n z^n$ in $\mathscr H_q$, we have 
\beqn
\|f\|^2_{\mathscr H_q} = \sum_{n=0}^{\infty} \Big(|a_{n}|^2\, \frac{(q)_n}{(1)_n} + \sum_{v \in V_{\prec}}\|b_{n, v}\|^2_{l^2(V)}\,\frac{(n_v+q+1)_n}{(n_v+2)_n}\Big),
\eeqn
where $f_n=a_{n}e_{\rootb} + \sum_{v \in V_{\prec}} b_{n, v} \in \ker \mathscr M^*_{z, q}$ is an orthogonal decomposition with $a_{n} \in \mathbb C$ and $b_{n, v} \in l^2(\mathsf{Chi}(v)) \ominus \langle \lambdab^v\rangle$ for every $n \in \mathbb N.$
Thus $\mathscr H_q$ is contained in $\mathscr H_1.$
\end{corollary}
\begin{proof}
%%Let $\mathscr T = (V, \mathcal E)$ be a locally finite, rooted directed tree of finite branching index and 
%Consider the Hilbert space $\mathscr K_q$ of $E$-valued formal power series given by
%\beqn
%\Big\{F(z)=\sum_{n=0}^{\infty} f_n z^n : \{f_n\}_{n \in \mathbb N} \subseteq E~\mbox{such~that~}\sum_{n=0}^{\infty}\|z^nf_n\|^2 < \infty \Big\}
%\eeqn
%endowed with the norm $\|f\|^2:=\sum_{n=0}^{\infty}\|z^nf_n\|^2.$ Since any $f \in \mathscr H_q$ can be rewritten as the orthogonal sum $f=\oplus_{n=0}^{\infty}f_n z^n$ for a unique sequence $\{f_n\}_{n=0}^{\infty}$ in $E$ (Proposition \ref{S-c-a-kernel-dim1}(iv)), one can define a unitary transformation
%$U : \mathscr H_q \rar \mathscr K_q$ by $Uf=F,$ where $F(z)=\sum_{n=0}^{\infty} f_n z^n.$ 
%%Arguing as in the proof of Proposition  \ref{S-c-a-kernel-dim1}, one can see that  $\{S^n_{\lambda, q}f_n\}_{n \in \mathbb N}$ is a mutually orthogonal sequence. Hence $U$ defines a unitary transformation.
%Further, since $\mathscr M_{z, q}f=\oplus_{n=1}^{\infty}z^nf_{n-1},$ $U\mathscr M_{z, q}f=G$ with $G(z)=\sum_{n=1}^{\infty} f_{n-1} z^n=z\sum_{n=1}^{\infty} f_{n-1} z^{n-1}z=zF(z).$ It follows that 
%\beqn
%U\mathscr M_{\lambda, q}f=M_{z, q}Uf,
%\eeqn
%where $M_{z, q}$ denotes operator of multiplication by $z$ on $\mathscr K_q.$ In view of Proposition \ref{S-c-a-kernel-dim1}, in the above calculations, 
%one can replace $l^2(V)$ by $\mathscr H_q$ and $S_{\lambda, q}$ by $\mathscr M_{z, q}$ respectively. It follows that
%$\|f\|_{\mathscr H_q}=\|F\|_{\mathscr K_q}=\sum_{n=0}^{\infty}\|S^n_{\lambda, q}f_n\|^2$ for every $f \in \mathscr H_q.$
To see the first part, in view of Proposition \ref{S-c-a-kernel-dim1}(iv), it suffices to check that for  $f_n=a_{n}e_{\rootb} + \sum_{v \in V_{\prec}} b_{n, v} \in \ker \mathscr M^*_{z, q},$
\beq
\label{formula-mono}
\|z^nf_n\|^2 = |a_{n}|^2\, \frac{(q)_n}{(1)_n} + \sum_{v \in V_{\prec}}\|b_{n, v}\|^2_{l^2(V)}\,\frac{(n_v+q+1)_n}{(n_v+2)_n}.
\eeq
However, by Proposition \ref{S-c-a-kernel-dim1}, $$\|z^nf_n\|^2 = \|S^n_{\lambda, q}f_n\|^2 = |a_n|^2  \|S^n_{\lambda, q}e_{\rootb}\|^2 + \sum_{v \in V_{\prec}} \|S^n_{\lambda, q}b_{n, v}\|^2,$$ and hence \eqref{formula-mono}
is immediate from Lemma \ref{formula-moments}.
%This is immediate from \beq \label{formula} S^{*n}_{\lambda, q} S^n_{\lambda, q} e_v=\frac{(n_v +q)_k}{(n_v+1)_k}\,e_v~(n \in \mathbb N),\eeq which can be derived as in the proof of Proposition \ref{S-c-a-kernel-dim1}.
The remaining part follows from the inequality $
\sum_{n=0}^{\infty} \Big(|a_{n}|^2 + \sum_{v \in V_{\prec}}\|b_{n, v}\|^2_{l^2(V)}\Big) \Le \|f\|^2_{\mathscr H_q}$ for every $f \in \mathscr H_q.$ 
\end{proof}
%\begin{remark} \label{D-norm}
%Note that \beqn
%\|f\|^2_{\mathscr H_2} &=& 
%%\sum_{n=0}^{\infty} \Big(|a_{n}|^2\, \frac{(2)_n}{(1)_n} + \sum_{v \in V_{\prec}}\|b_{n, v}\|^2_{E}\,\frac{(n_v+3)_n}{(n_v+2)_n}\Big)=
%\sum_{n=0}^{\infty} \Big((n+1)|a_{n}|^2 + \sum_{v \in V_{\prec}} \Big(1+ \frac{n}{n_v+2}\Big) \|b_{n, v}\|^2_{E}\Big) \\
%&=& \|f\|^2_{\mathscr H_1} + \sum_{n=0}^{\infty} n\Big( |a_n|^2 + \sum_{v \in V_{\prec}} \frac{\|b_{n, v}\|^2_{E}}{n_v+2}\Big).
%\eeqn
%\end{remark}

Let $\kappa : \mathbb D \times \mathbb D \rar B(E)$ be a positive definite kernel such that $\kappa(z, w)$ is invertible for every $z , w \in \mathbb D$ and let $\mathscr H$ be the reproducing kernel Hilbert space associated with $\kappa.$ Following \cite{A-M}, we say that $\mathscr H$ has {\it complete Pick property} if there exists a positive definite function $F : \mathbb D \times \mathbb D \rar B(E)$ such that
\beqn
I_E - \kappa(z, w)^{-1} = F(z, w)~(z, w \in \mathbb D).
\eeqn 

\begin{corollary}
Let $\mathscr H_q$ be a Dirichlet space associated with $\mathscr T=(V, \mathcal E)$. Then $\mathscr H_q$ has complete Pick property.
\end{corollary}
\begin{proof} By Proposition \ref{S-c-a-kernel-dim1}(ii), $\kappa_{\mathscr H_q}(z, w)$ is orthogonal direct sum of finitely many positive definite kernels of the form
\beqn \kappa_{k, l}(z, w)=\sum_{n=0}^{\infty}  \frac{(k)_n}{(l)_n} {z^n \overline{w}^n}P_{k, l},
\eeqn
where $P_{k, l}$ is a non-zero orthogonal projection and $k, l$ are positive integers such that $l \Ge k$. 
Thus it suffices to check that $1 -\frac{1}{\kappa_{k, l}(z, w)}$ is a positive definite kernel. In view of \cite[Theorem 7.33 and Lemma 7.38]{A-M}, it is enough to verify that
\beqn
\Big(\frac{(k)_n}{(l)_n}\Big)^2 \Le \frac{(k)_{n-1}}{(l)_{n-1}}\frac{(k)_{n+1}}{(l)_{n+1}}~\mbox{for every~}n \in \mathbb N.
\eeqn
However, this is equivalent to ${(l+n)(k+n-1)}{} \Le {(l+n-1)(k+n)}{}~(n \in \mathbb N),$ which is true whenever $l \Ge k$.
\end{proof}

\subsection{The vector-valued Dirichlet space $\mathscr H_2$}

It turns out that $\mathscr H_2$ can be identified with a vector-valued Dirichlet space.
Let us first reproduce from \cite{O} the definition of vector-valued Dirichlet spaces. 

Let $E$ be a Hilbert space and $\mu$ be a positive $B(E)$-valued measure on the unit circle $\mathbb T.$ The {\it Poisson integral} $P[\mu] : \mathbb D \rar B(E)$ of $\mu$ is defined by
\beqn
P[\mu](z) := \int_{\mathbb T} \frac{1-|z|^2}{|e^{i \theta}-z|^2} d\mu(e^{i \theta})~(z \in \mathbb D).
\eeqn
For an analytic function $f : \mathbb D \rar E$ of the form $f(z)=\sum_{n=0}^{\infty}f_n z^n$ with $\{f_n\}_{n \in \mathbb N} \subseteq E$, set
\beq \label{Diri-n}
\|f\|^2_{\mu} :=\sum_{n=0}^{\infty}\|f_n\|^2_E + \int_{\mathbb D} \inp{P[\mu](z)f'(z)}{f'(z)}_E\, dA(z), 
\eeq
where $dA$ denotes normalized area measure on $\mathbb D.$
The {\it $E$-valued Dirichlet space} is defined as
\beqn
D(\mu):=\{f : \mathbb D \rar  E ~|~ f~\mbox{is an analytic function such that~} \|f\|_{\mu} < \infty\}.
\eeqn

\begin{proposition} 
Let $\mathscr H_2$ be a Dirichlet space associated with $\mathscr T=(V, \mathcal E)$ and let $\mathscr M_{z, 2}$ be the operator of multiplication by $z$ on $\mathscr H_2.$
%Let $\mathscr T = (V, \mathcal E)$ be a locally finite, rooted directed tree of finite branching index and consider the weighted shift $S_{\lambda, 2}$ on $\mathscr T$.  
Define a
positive $B(\mathscr E)$-valued measure $\mu_{{}_{\mathscr T}}$ given by
\beq \label{measure}
d\mu_{{}_{\mathscr T}}(e^{i \theta}) \Big(a e_{\rootb} + \sum_{v \in V_{\prec}} b_{v}\Big):= \Big(a e_{\rootb} + \sum_{v \in V_{\prec}} \frac{b_{v}}{n_v+2}\Big)d\sigma(e^{i \theta}),
\eeq
where $\mathscr E:=\ker \mathscr M^*_{z, 2}$,
$a \in \mathbb C$, $b_{v} \in l^2(\mathsf{Chi}(v)) \ominus \langle \lambdab^v\rangle$, and
$d\sigma$ denotes the normalized arclength measure on unit circle $\mathbb T$.
Then $\mathscr H_2$  equals the $\mathscr E$-valued Dirichlet space $D(\mu_{{}_{\mathscr T}})$ with equality of norms.
\end{proposition}
\begin{proof}
We claim that if $f \in \mathscr H_2$ then $\|f\|_{\mathscr H_2}=\|f\|_{\mu_{{}_{\mathscr T}}}$. 
Note that by Corollary \ref{D-norm}, \beq \label{eq-norm-f}
\|f\|^2_{\mathscr H_2} &=& 
%\sum_{n=0}^{\infty} \Big(|a_{n}|^2\, \frac{(2)_n}{(1)_n} + \sum_{v \in V_{\prec}}\|b_{n, v}\|^2_{E}\,\frac{(n_v+3)_n}{(n_v+2)_n}\Big)=
\sum_{n=0}^{\infty} \Big((n+1)|a_{n}|^2 + \sum_{v \in V_{\prec}} \Big(1+ \frac{n}{n_v+2}\Big) \|b_{n, v}\|^2_{E}\Big) \notag \\
&=& \|f\|^2_{\mathscr H_1} + \sum_{n=0}^{\infty} n\Big( |a_n|^2 + \sum_{v \in V_{\prec}} \frac{\|b_{n, v}\|^2_{E}}{n_v+2}\Big).
\eeq
In view of Proposition \ref{S-c-a-kernel-dim1}, it suffices to verify 
$\|f_n\|_{\mathscr H_2}=\|f_n\|_{\mu_{{}_{\mathscr T}}}$, where $$f_n(z)=\left(a e_{\rootb} + \sum_{v \in V_{\prec}} b_{v}\right)z^n~(n \Ge 1).$$ 
However, in the light of \eqref{Diri-n} and \eqref{eq-norm-f}, it is enough to verify that 
\beqn
n\Big( |a|^2 + \sum_{v \in V_{\prec}} \frac{\|b_{v}\|^2_{l^2(V)}}{n_v+2}\Big) = \int_{\mathbb D} \inp{P[\mu_{{}_{\mathscr T}}](z)f'_n(z)}{f'_n(z)}_{E}\, dA(z),
\eeqn
where $f'_n(z)=\Big(a e_{\rootb} + \sum_{v \in V_{\prec}} b_{v}\Big)n z^{n-1}$ for integers $n \geqslant 1.$ Note first that for $z \in \mathbb D,$
\beq \label{inde-m}
P[\mu_{{}_{\mathscr T}}](z) \Big(a e_{\rootb} + \sum_{v \in V_{\prec}} b_{v}\Big) &=& P[\sigma](z)\Big(a e_{\rootb} + \sum_{v \in V_{\prec}} \frac{b_{v}}{n_v+2}\Big) \notag \\
&=&a e_{\rootb} + \sum_{v \in V_{\prec}} \frac{b_{v}}{n_v+2}.
\eeq
Since $\|g\|_{\mathscr H_2}=\|g\|_{l^2(V)}$ for every $g \in \mathscr E$ (Proposition \ref{S-c-a-kernel-dim1}),
it now follows that
\beqn
&& \int_{\mathbb D} \inp{P[\mu_{{}_{\mathscr T}}](z)f'_n(z)}{f'_n(z)}_{E}\, dA(z) \\ &=& \int_{\mathbb D} n^2|z^{n-1}|^2 \Big \langle a e_{\rootb} + \sum_{v \in V_{\prec}} \frac{b_{v}}{n_v+2}, ~{a e_{\rootb} + \sum_{v \in V_{\prec}} b_{v}}\Big \rangle_{E}\, dA(z) \\
&=&\Big( |a|^2 + \sum_{v \in V_{\prec}} \frac{\|b_{v}\|^2_{\mathscr H_2}}{n_v+2}\Big) \int_{\mathbb D} n^2|z^{n-1}|^2  dA(z) \\
%&=& \Big( |a|^2 + \sum_{v \in V_{\prec}} \frac{\|b_{v}\|^2_{l^2(V)}}{n_v+2}\Big) (\|z^{n}\|^2_{D_2} - 1) \\
&=& n\Big( |a|^2 + \sum_{v \in V_{\prec}} \frac{\|b_{v}\|^2_{l^2(V)}}{n_v+2}\Big).
\eeqn
Thus the claim stands verified.
The claim implies in particular that $\mathscr H_2 \subseteq D(\mu_{{}_{\mathscr T}})$. Since $E$-valued analytic polynomials are dense in $D(\mu_{{}_{\mathscr T}})$ (\cite[Corollary 3.1]{O}),
we must have $\mathscr H_2 = D(\mu_{{}_{\mathscr T}})$.
\end{proof}

As in the classical case \cite[Corollary 1.4.3]{EKMR}, any Dirichlet space $\mathscr H_2$ associated with $\mathscr T$ admits the conformal invariance property.
\begin{corollary}
Let $\mathscr H_2$ be a Dirichlet space associated with $\mathscr T=(V, \mathcal E)$ and
let $\phi$ be an automorphism of the unit disc. Then, for every $f \in  \mathscr H_2,$ $f \circ \phi \in \mathscr H_2$ and
\beqn \int_{\mathbb D} \inp{P[\mu_{\mathscr T}](z)(f \circ \phi)'(z)}{(f \circ \phi)'(z)}_E\, dA(z)=\int_{\mathbb D} \inp{P[\mu_{\mathscr T}](z)f'(z)}{f'(z)}_E\, dA(z),
\eeqn
where $\mu_{\mathscr T}$ is as given in \eqref{measure}.
\end{corollary}
\begin{proof}
Let $f \in \mathscr H_2$. 
%and let $f(z) = \sum_{n=0}^{\infty}f_n z^n$ for a sequence $\{f_n\}_{n \in \mathbb N}$ in $E.$
%It suffices to check that 
Note that by change of variables, \beqn && \int_{\mathbb D} \inp{P[\mu_{\mathscr T}](z)(f \circ \phi)'(z)}{(f \circ \phi)'(z)}_E\, dA(z) \\ &=& \int_{\mathbb D}  \inp{P[\mu_{\mathscr T}](z)f'(\phi(z))}{f'(\phi(z))}_E\, |\phi'(z)|^2 dA(z) 
%\\ &=& \int_{\mathbb D}  \inp{P[\mu_{\mathscr T}](\phi^{-1}(w))f'(w)}{f'(w)}_E\, dA(w) 
\\ &\overset{\eqref{inde-m}}=& \int_{\mathbb D}  \inp{P[\mu_{\mathscr T}](w)f'(w)}{f'(w)}_E\, dA(w),
\eeqn
which completes the proof of the second part. The first part now follows from Littlewood's Theorem \cite[Chapter 1]{S}.
\end{proof}
%\begin{remark}
%It is clear from Theorem \ref{solution} that there are infinitely many directed trees for which the associated Dirichlet shifts are not unitarily equivalent. This implies in particular that there are infinitely many Dirichlet spaces $\mathscr H_2$ with conformal invariance property.
%\end{remark}

%\begin{corollary} For $j=1, 2$, let $\mathscr T_i = (V_j, \mathcal E_j)$ be a locally finite, rooted directed tree of finite branching index and let 
%$\mathcal G^{(j)}_n:=\{v \in V_j : v \in \childn{n}{\rootb}\}~(n \in \mathbb N).$ For $j=1, 2$, let
%$\mathscr M^{(j)}_{z, 2}$ denote the operator of multiplication by $z$ on the $\mathscr T_j$-analog of the classical Dirichlet space $D_2$ and let $\mathscr E_j=\ker (\mathscr M^{(j)}_{z, 2})^*$. Then 
%%the following statements are equivalent:
%%\begin{itemize}
%%\item [(i)] 
%$\mathscr M^{(1)}_{z, 2}$ is unitarily equivalent to $\mathscr M^{(2)}_{z, 2}$ if and only if
%%\item [(ii)]
%for every $n \in \mathbb N,$ 
%%$\mbox{card}(V^{(1)}_{\prec} \cap \mathcal G^{(1)}_n) =\mbox{card}(V^{(2)}_{\prec} \cap \mathcal G^{(2)}_n)$ and 
%$$\sum_{v \in V^{(1)}_{\prec} \cap \mathcal G^{(1)}_n}\Big(\mbox{card}(\child{v})-1\Big)=\sum_{v \in V^{(1)}_{\prec} \cap \mathcal G^{(2)}_n}\Big(\mbox{card}(\child{v})-1\Big).$$
%%\end{itemize}
%\end{corollary}

We present below a proof of the special case $q=2$ of Theorem \ref{solution}, which exploits the theory of vector-valued Dirichlet spaces \cite{O}.
\begin{proof}[Proof of Theorem \ref{solution} (Case $q=2$)]
For $j=1, 2$, let $\mathscr H^{(j)}_2$ be the Dirichlet space associated with $\mathscr T_j=(V_j, \mathcal E_j)$ and let $\mathscr M^{(j)}_{z, 2}$ be the operator of multiplication by $z$ on $\mathscr H^{(j)}_2.$
In view of Proposition \ref{S-c-a-kernel-dim1}, it suffices to check that $\mathscr M^{(1)}_{z, 2}$ is unitarily equivalent to $\mathscr M^{(2)}_{z, 2}$ if and only if \eqref{constant-g} holds for every $n \in \mathbb N.$
By the preceding proposition and \cite[Theorem 4.2]{O},
the multiplication operator $\mathscr M^{(1)}_{z, 2}$ on $D(\mu_{{}_{\mathscr T_1}})$ is unitarily equivalent to the multiplication operator $\mathscr M^{(2)}_{z, 2}$ on $D(\mu_{{}_{\mathscr T_2}})$ if and only if 
there exists a unitary map $U : \mathscr E_1 \rar \mathscr E_2$ such that $\mu_{{}_{\mathscr T_1}}(A) = U^*\mu_{{}_{\mathscr T_2}}(A)U$ for every Borel subset $A$ of the unit circle $\mathbb T$ (see \eqref{measure}). However, by \eqref{measure}, this happens if and only if the diagonal matrices 
$\oplus_{v \in V^{(1)}_{\prec}} \alpha_v I_{\beta_{v, 1}}$ and $\oplus_{v \in V^{(2)}_{\prec}} \alpha_v I_{\beta_{v, 2}}$ are unitarily equivalent, where 
$I_m$ denotes the $m \times m$ identity matrix and
\beqn \alpha_v:= \frac{1}{n_v+2}, \quad \beta_{v, j}:= \mbox{card}(\child{v})-1~(j=1, 2).\eeqn 
The later one holds if and only if $\oplus_{v \in V^{(1)}_{\prec}} \alpha_v I_{\beta_{v, 1}}$ and $\oplus_{v \in V^{(2)}_{\prec}} \alpha_v I_{\beta_{v, 2}}$ have the same eigenvalues counted with multiplicity.
%Note that if $v \in V^{(1)}_{\prec} \cap \mathcal G^{(1)}_n$ then $n_v=n$ and hence 
However, $\sum_{v \in V^{(j)}_{\prec} \cap \mathcal G^{(1)}_n}\Big(\mbox{card}(\child{v})-1\Big)$ is the multiplicity of the eigenvalue $\frac{1}{n+2}$ for $j=1, 2.$
The desired equivalence is now immediate.
\end{proof}

\section{Bergman Spaces Associated with Directed Trees}

%\section{Proof of Theorem \ref{solution}}

%\begin{remark} \label{rem-D-B}
It is evident from the proof of Proposition \ref{S-c-a-kernel-dim1} that one can also associate a functional Hilbert space, say, $\mathscr H_{\textendash q}$ with the Cauchy dual $S_{\lambda', q}$ of the Dirichlet shift $S_{\lambda, q}.$ 
In particular, $\mathscr H_{\textendash q}$ is a reproducing kernel Hilbert space associated with the kernel $\kappa_{\mathscr H_{\textendash q}}$ given by
 \beqn
\kappa_{\mathscr H_{\textendash q}}(z, w) &=& \sum_{n=0}^{\infty}\frac{(q)_n}{(1)_n}~ {z^n \overline{w}^n} ~P_{\langle e_\rootb \rangle} \\  &+& 
\sum_{v \in V_{\prec}} \sum_{n=0}^{\infty}  \frac{(n_v +q+1)_n}{(n_v + 2)_n} {z^n \overline{w}^n}~P_{l^2(\child{v}) \ominus \langle \lambdab^v \rangle}~(z, w \in \mathbb D).
\eeqn
This has been recorded in \cite[Proposition 5.1.8]{CPT}. 
It is equally clear that $S_{\lambda', q}$ is unitarily equivalent to the operator of multiplication by $z$ on $\mathscr H_{\textendash q}$. 
%\end{remark}

\begin{definition} 
Let $q$ be an integer bigger than $1$. 
Let $\mathscr T = (V, \mathcal E)$ be a leafless, locally finite rooted directed tree. 
%of finite branching index. 
%We refer to the space $\mathscr H_q$, as constructed in Proposition \ref{S-c-a-kernel-dim1}, as the {\it Dirichlet space associated with $\mathscr T$.}
We refer to $\mathscr H_{\textendash q}$ as the {\it Bergman space associated with $\mathscr T$.}
\end{definition}
\begin{remark} \label{c-Bergman}
Note that in case $\mathscr T$ is a rooted directed tree without any branching vertex, $\mathscr H_{\textendash q}$ is nothing but the {\it classical Bergman space} $B_q~(q \Ge 2)$, that is, the reproducing kernel Hilbert space associated with $$\sum_{n=0}^{\infty}\frac{(q)_n}{(1)_n}~ {z^n \overline{w}^n}~(z, w \in \mathbb D)$$
(refer to \cite{HKZ} for the basic theory of classical Bergman spaces).  
\end{remark}

%The proof of Theorem \ref{solution} relies on the following lemma.
\begin{lemma} \label{B-norm}
Let $\mathscr H_{\textendash q}$ be a Bergman space associated with $\mathscr T=(V, \mathcal E)$ and let $\mathscr M_{z, \textendash q}$ be the operator of multiplication by $z$ on $\mathscr H_{\textendash q}.$
Then for $f(z)=\sum_{n=0}^{\infty}f_n z^n$ in $\mathscr H_{\textendash q}$ 
with $f_n=a_{n}e_{\rootb} + \sum_{v \in V_{\prec}} b_{n, v} \in \ker \mathscr M^*_{z, \textendash q}$, $a_{n} \in \mathbb C$ and $b_{n, v} \in l^2(\mathsf{Chi}(v)) \ominus \langle \lambdab^v\rangle$ for every $n \in \mathbb N,$
we have the following:
\begin{itemize}
\item[(i)]
$
\|f\|^2_{\mathscr H_{\textendash q}} = \sum_{n=0}^{\infty} \Big(|a_{n}|^2\, \frac{(1)_n}{(q)_n} + \sum_{v \in V_{\prec}}\|b_{n, v}\|^2_{l^2(V)}\,\frac{(n_v+2)_n}{(n_v+q+1)_n}\Big).
$
\item[(ii)]  if $q \Ge 2,$ then $$
\|f\|^2_{\mathscr H_{\textendash q}} =\int_{\mathbb D} \inp{d\nu_{{}_{\mathscr T}}(z)f(z)}{f(z)}=\sum_{n=0}^{\infty}\int_{\mathbb D} |z^n|^2 \inp{d\nu_{{}_{\mathscr T}}(z)f_n}{f_n},$$
where \beq \label{measure-0}
d\nu_{{}_{\mathscr T}}(z) \Big(a e_{\rootb} + \sum_{v \in V_{\prec}} b_{v}\Big):= \Big(a w_0(z) e_{\rootb} + \sum_{v \in V_{\prec}} w_{n_v+1}(z)b_v \Big)dA(z)
\eeq
with 
$dA$ denoting the normalized area measure on unit disc $\mathbb D$ and $$w_l(z)={(l+1)\cdots(l+q-1)} \displaystyle \sum_{i=1}^{q-1}\frac{|z|^{2(i + l-1)}}{\displaystyle
\prod_{1 \Le j \neq i \Le q-1}(j -i)}~(z \in \mathbb D, ~l \in \mathbb N).$$
\end{itemize}
In particular, $\mathscr H_{\textendash q}$ equals $L^2_a(\nu_{{}_{\mathscr T}})$, that is, the closure of $E$-valued analytic polynomials in $L^2(\nu_{{}_{\mathscr T}})$, with equality of norms.
\end{lemma}
\begin{proof}
The first part can be obtained along the lines of the proof of Corollary \ref{D-norm}. The first equality in the second part may be deduced from (i) and \cite[Corollary 3.8]{AC} while the second one is an immediate consequence of the fact that $$\int_{\mathbb D} z^n \overline{z}^m w(|z|)dA(z)=\delta_{mn}~(m, n \in \mathbb N)$$ for any continuous function $w : [0, 1] \rar (0, \infty).$
\end{proof}
%\begin{remark}
%The spectrum of $\mathscr M_{z, q}$ is contained in the closed unit disc $\overline{\mathbb D}$.
%\end{remark}
The Bergman spaces $\mathscr H_{\textendash q}$ can be realized as Hilbert modules over the polynomial ring $\mathbb C[z]$ with module action given by 
\beqn
(p, f) \in \mathbb C[z] \times \mathscr H_{\textendash q} \longmapsto p(z)f \in \mathscr H_{\textendash q}.
\eeqn
A family of Hilbert modules $\mathcal H^{(\eta, Y)}_N$ (vector-valued analogs of the reproducing kernel Hilbert spaces associated with the kernel $\frac{1}{(1-z\overline{w})^{q}},$ $q >0$) have been introduced in \cite{KM} in the context of classification of homogeneous operators within the class of Cowen-Douglas operators (see the discussion prior to \cite[Theorem 3.1]{KM} for the exact description of the spaces $\mathcal H^{(\eta, Y)}_N$). 
Recall that a bounded linear operator $T$ on $\mathcal H$ is {\it homogeneous} if  for every automorphism $\phi$ of the open unit disc $\mathbb D,$ $\phi(T)$ is unitarily equivalent to 
$T$. 
%\begin{remark}
Note that the assumption that $\phi(T)$ is well-defined as a  bounded linear operator on $\mathcal H$ is a part of the definition of the homogeneous operator. Also, the multiplication operator $\mathscr M_{z, \textendash q}$ on the classical Bergman space $B_q$ provides an example of homogeneous operator \cite[Theorem 5.2]{BM}.
%\end{remark}

We say that the Hilbert module $\mathscr H_{\textendash q}$ over $\mathbb C[z]$ is {\it homogeneous} if $\mathscr M_{z, \textendash q}$ is homogeneous.
The natural question arises here is that for which directed trees $\mathscr T$, $\mathscr H_{\textendash q}$ is a homogeneous Hilbert module ? The second question arises is whether
 $\mathscr H_{\textendash q}$ and $\mathcal H^{(\eta, Y)}_N$ are unitarily equivalent as Hilbert modules over $\mathbb C[z]$ ? Both these questions can be answered with the help of \cite[Theorem 4.1]{KM} and the following general fact.

\begin{proposition} \label{homogeneous}
Let $\mathscr H_{\textendash q}$ be a Bergman space associated with the directed tree $\mathscr T=(V, \mathcal E)$ of finite branching index. 
%and let $\mathscr M_{z, q}$ be the operator of multiplication by $z$ on $\mathscr H_{\textendash q}.$
Then 
%Then for every automorphism $\phi$ of the open unit disc $\mathbb D,$ $\phi(\mathscr M_{z, q})$ is well-defined as a bounded linear operator on $\mathcal H'_q$. Further, 
the following statements are equivalent: 
\begin{itemize}
\item[(i)] $\mathscr H_{\textendash q}$ is a homogeneous Hilbert module. 
\item[(ii)] $\mathscr T$ is isomorphic to $\mathbb N.$
\end{itemize}
\end{proposition}
\begin{proof} Let $\phi$ be an  automorphism of the open unit disc $\mathbb D.$
Since $\mathscr M_{z, \textendash q}$ is of spectral radius $1$ (Lemma \ref{formula-moments}) and since $\kappa({\cdot, w})g$ provides eigenvector for $\mathscr M^*_{z, \textendash q}$ with corresponding eigenvalue $\overline{w} \in \mathbb D$, the spectrum of $\mathscr M_{z, \textendash q}$ equals the closed unit disc $\overline{\mathbb D}.$
In particular,  $\mathscr M_{\phi, q}:=\phi(\mathscr M_{z, \textendash q})$ makes sense. 

In case $\mathscr T$ is isomorphic to $\mathbb N$, the implication (ii) $\Longrightarrow$ (i) is immediate from
\cite[Theorem 5.2]{BM}. 

To see that (i) implies (ii), suppose there exists a unitary $U : \mathscr H_{\textendash q} \rar \mathscr H_{\textendash q}$ such that $U\mathscr M_{z, \textendash q}=\mathscr M_{\phi, q}U.$  Let $\{g_j : j=1, \cdots, d\}$ be an orthonormal basis for $E:=\ker \mathscr M^*_{z, \textendash q}$, where $g_1:=e_{\rootb}$ and $d$ is finite. One may conclude from \eqref{moment} and \eqref{orthog} that for $1 \Le i \neq j \Le d,$ the sequences $\{z^ng_i\}_{n \in \mathbb N}$ and $\{z^ng_j\}_{n \in \mathbb N}$ are mutually orthogonal.
This yields the decomposition 
$\mathscr H_{\textendash q} = \oplus_{j=1}^d \mathscr H_{j},$ where
\beqn \mathscr H_j = \bigvee \{z^ng_j :n \in \mathbb N\}~(j=1, \cdots, d).
\eeqn
Since $\mathscr H_1, \cdots, \mathscr H_d$ are $z$-invariant subspaces of $\mathscr H_{\textendash q},$ $\mathscr M_{z, \textendash q}=\oplus_{j=1}^d \mathscr M^{(j)}_{z, \textendash q},$ where 
$\mathscr M^{(j)}_{z, \textendash q}=\mathscr M_{z, \textendash q}|_{\mathscr H_j}$ for $j=1, \cdots, d.$ On the other hand, by \cite[Corollary 5.6]{CT},
the Hilbert space adjoint of $\mathscr M_{z, \textendash q}$ belongs to the Cowen-Douglas class $B_d(\mathbb D)$ (the reader is referred to \cite{CD} for the definition of Cowen-Douglas class $B_d(\cdot)$). Also, the Hilbert space adjoint of $\mathscr M^{(j)}_{z, \textendash q}$ belongs to $B_1(\mathbb D)$ for every $j=1, \cdots, d.$
Since the classification of homogeneous operators in 
$B_d(\mathbb D)$
is the same as that of the corresponding homogeneous holomorphic Hermitian bundles defined on
$\mathbb D$ \cite[Section 1]{BiM},
we may conclude from \cite[Corollary 2.1]{KM} that $\mathscr M^{(j)}_{z, \textendash q}$ must be homogeneous for every $j=1, \cdots, d.$ 
However, by \cite[List 3.1]{BM}, this is possible only if  for every $j=1, \cdots, d,$ there exist $a_j >0$ such that $\mathscr M^{(j)}_{z, \textendash q}$ is unitarily equivalent to the operator $M^{(j)}_z$ of multiplication by $z$ on a reproducing kernel Hilbert space $H_j$ associated with the kernel $\frac{1}{(1-z\overline{w})^{a_j}}~(z, w \in \mathbb D).$ 
%If possible, assume that $V_{\prec} \neq \emptyset,$ so that $d \geq 2.$ 
For $j=1, \cdots, d$, let $U_j : H_j \rar \mathscr H_j$ be the unitary operator such that $$\mathscr M^{(j)}_{z, \textendash q}U_j =U_j M^{(j)}_z.$$ Note that $U_j$ maps $\ker (M^{(j)}_z)^*$ onto $\ker (\mathscr M^{(j)}_{z, \textendash q})^*$. 
Note that $g:=U_j (1 / \|1\|)  \in \ker (\mathscr M^{(j)}_{z, \textendash q})^*$. 
If $d=1$ then clearly $\mathscr T$ is isomorphic to $\mathbb N$. Suppose the contrary. Then
$g \neq e_{\rootb}$ for some $j=1, \cdots, d.$
Thus $g=\sum_{w \in \child{v}}\alpha_w e_w \in l^2(\mathsf{Chi}(v)) \ominus \langle \lambdab^v\rangle$ for some $v \in V_{\prec}$. Hence, for any integer $k \geqslant 1,$ by Lemma \ref{formula-moments}, 
\beqn
\|(\mathscr M^{(j)}_{z, \textendash q})^{k}g\|^2=\sum_{w \in \child{v}}|\alpha_w|^2  \frac{(n_w +1)_{k}}{(n_w+q)_{k}}=\frac{(n_v +2)_{k}}{(n_v+q+1)_{k}}.
\eeqn
On the other hand,
\beqn
\|(\mathscr M^{(j)}_{z, \textendash q})^{k}g\|^2=\|(M^{(j)}_z)^{k}(1/\|1\|)\|^2=
\frac{(1)_{k}}{(a_j)_{k}},
\eeqn
After combining last two equations for $k=1, 2$, we obtain $n_v+1=0$, which is absurd. Hence $d=1$, and   
we obtain (ii).
%For $f(z)=p(z)g$ in $\mathscr H_{\textendash q}$ 
%with $p$ is a scalar-valued analytic polynomial in $z$, $g=a e_{\rootb} + \sum_{v \in V_{\prec}} b_{v} \in \ker \mathscr M^*_{z, q}$, $a \in \mathbb C$ and $b_{v} \in l^2(\mathsf{Chi}(v)) \ominus \langle \lambdab^v\rangle$, define $U_{\phi}f$ by
%\beqn
%(U_{\phi}f)(z) &=& \frac{p \circ \phi(z)}{{\displaystyle
%\prod_{1 \Le j \neq i \Le q-1}(j -i)}} \Big( a  (q-1)! \displaystyle \sum_{i=1}^{q-1}{D_{\phi, i -1}(z)}e_{\rootb} \\ &+& \sum_{v \in V_{\prec}} {(n_v+2)\cdots(n_v+q)} \displaystyle \sum_{i=1}^{q-1}{D_{\phi, i +n_v}(z)} b_v\Big),
%\eeqn
%where $D_{\phi, k}(z) = \Big(\frac{\phi(z)}{z}\Big)^{2k} \frac{1}{(\phi^{-1})'(\phi(z))^2}$ is a meromorphic function with removable singularity at the origin. Clearly,
%$Uf$ defines an $E$-valued holomorphic function on $\mathbb D$. We contend that $U_{\phi}f$ belongs to $\mathscr H_{\textendash q}$ with 
%$\|Uf\|_{\mathscr H_{\textendash q}}=\|f\|_{\mathscr H_{\textendash q}}.$
\end{proof}

The proof of the preceding proposition actually yields the following general fact.
\begin{corollary} \label{coro-h}
Let $S_{\lambda}$ be a left-invertible, homogeneous weighted shift on $\mathscr T$ with finite dimensional cokernel $E:=\ker S^*_{\lambda}.$ If $\{g_j : j=1, \cdots, d\}$ is an orthonormal basis of $E$ such that  
for $1 \Le i \neq j \Le d,$ the sequences $\{S^n_{\lambda}g_i\}_{n \in \mathbb N}$ and $\{S^n_{\lambda}g_j\}_{n \in \mathbb N}$ are mutually orthogonal, then there exists $a > 0$ such that $S_{\lambda}$ is unitarily equivalent to the operator of multiplication by $z$ on a reproducing kernel Hilbert space associated with the kernel $\frac{1}{(1-z\overline{w})^{a}}~(z, w \in \mathbb D).$
\end{corollary}

We do not know whether a rooted directed tree, which is non-isomorphic to $\mathbb N$, supports a homogeneous weighted shift.  

\section{Classification of Dirichlet Shifts}
For the sake of convenience, we reproduce the statement of Theorem \ref{solution} from Section 2.
\begin{theorem} 
Let $q$ be an integer bigger than $1.$
For $j=1, 2$, let
$S^{(j)}_{\lambda, q}$ be the Dirichlet shift on rooted directed tree $\mathscr T_j= (V_j, \mathcal E_j)$ with $\rootb_j$, let
$\mathcal G^{(j)}_n:=\{v \in V_j : v \in \childn{n}{\rootb_j}\}~(n \in \mathbb N),$ and $E_j=\ker (S^{(j)}_{\lambda, q})^*$. Then 
$S^{(1)}_{\lambda, q}$ is unitarily equivalent to $S^{(2)}_{\lambda, q}$ if and only if
for every $n \in \mathbb N,$ 
%$\mbox{card}(V^{(1)}_{\prec} \cap \mathcal G^{(1)}_n) =\mbox{card}(V^{(2)}_{\prec} \cap \mathcal G^{(2)}_n)$ and 
\beq \label{constant-g2} \sum_{v \in V^{(1)}_{\prec} \cap \mathcal G^{(1)}_n}\Big(\mbox{card}(\child{v})-1\Big)=\sum_{v \in V^{(1)}_{\prec} \cap \mathcal G^{(2)}_n}\Big(\mbox{card}(\child{v})-1\Big).\eeq
\end{theorem}
\begin{proof}
%[Proof of Theorem \ref{solution}]
%In view of the preceding discussion, we may assume that $q \Ge 2.$ 
Note that $S^{(1)}_{\lambda, q}$ is unitarily equivalent to $S^{(2)}_{\lambda, q}$ if and only if $S^{(1)}_{\lambda', q}$ is unitarily equivalent to $S^{(2)}_{\lambda', q}$. Hence, in view of Lemma \ref{B-norm}, it suffices to check that the multiplication operator $\mathscr M^{(1)}_{z, q}$ on $L^2_a(\nu_{{}_{\mathscr T_1}})$ is unitarily equivalent to the multiplication operator $\mathscr M^{(2)}_{z, q}$ on $L^2_a(\nu_{{}_{\mathscr T_2}})$ if and only if \eqref{constant-g2} holds for every $n \in \mathbb N$. 

Let $U : L^2_a(\nu_{{}_{\mathscr T_1}}) \rar L^2_a(\nu_{{}_{\mathscr T_2}})$ be a unitary map such that \beq \label{intw} U\mathscr M^{(1)}_{z, q} = \mathscr M^{(2)}_{z, q}U.\eeq Note that $U$ maps $\ker (\mathscr M^{(1)}_{z, q})^*$ unitarily onto $\ker (\mathscr M^{(2)}_{z, q})^*.$ For $g, h \in \ker (\mathscr M^{(1)}_{z, q})^*,$ define a signed finite measure $\mu_{g, h}$ on the open unit disc $\mathbb D$ by
\beqn
\mu_{g, h}(\sigma) = \int_{\sigma}\big\langle{\big (U^*d\nu_{{}_{\mathscr T_2}}(z)U - d\nu_{{}_{\mathscr T_1}}(z)\big)g}, {h}\big \rangle~\mbox{for a Borel subset}~\sigma~\mbox{of~}\mathbb D.
\eeqn 
%Since $U(\mathscr M^{(1)}_{z, q})^n = (\mathscr M^{(2)}_{z, q})^nU~(n \in \mathbb N),$ for $g \in E,$
%\beqn
%(Uz^ng)(w)=w^nUg.
%\eeqn
%It is now easy to see that \beq \label{unitary} (Uf)(z)=U(f(z))~(f \in L^2_a(\nu_{{}_{\mathscr T_1}}), ~z \in \mathbb D).\eeq 
Note that
\beqn
\int_{\mathbb D}z^n \overline{z}^m\big\langle{U^*d\nu_{{}_{\mathscr T_2}}(z)Ug}, {h}\big \rangle &=& \inp{z^nUg}{z^mUh}_{L^2_a(\nu_{{}_{\mathscr T_2}})} \overset{\eqref{intw}}= \inp{Uz^ng}{Uz^mh}_{L^2_a(\nu_{{}_{\mathscr T_2}})} \\ &=& \inp{z^ng}{z^mh}_{L^2_a(\nu_{{}_{\mathscr T_1}})} = \int_{\mathbb D}z^n \overline{z}^m\big\langle{d\nu_{{}_{\mathscr T_1}}(z)g}, {h}\big \rangle.
\eeqn
%- d\nu_{{}_{\mathscr T_1}}(z)
It follows that $\int_{\mathbb D}p(z, \overline{z})d\mu_{g, h}=0$ for any polynomial $p$ in $z$ and $\overline{z}$. By Stone-Weierstrass and Riesz Representation Theorems,  $\mu_{g, h}$ is identically $0.$ Since this holds for arbitrary choices of $g, h$ in $\ker (\mathscr M^{(1)}_{z, q})^*$, we conclude that 
\beq \label{uni-E}
U^*\nu_{{}_{\mathscr T_2}}(\sigma)U = \nu_{{}_{\mathscr T_1}}(\sigma)~\mbox{for every Borel subset}~\sigma~\mbox{of~}\mathbb D.
\eeq 
It is now clear from \eqref{measure-0} that for every Borel subset $\sigma$ of $\mathbb D$, $a \in \mathbb C$, and $b_{v} \in l^2(\mathsf{Chi}(v)) \ominus \langle \lambdab^v\rangle~(v \in V^{(1)}_{\prec})$, 
\beq \label{eqn} 
&& \nu_{{}_{\mathscr T_2}}(\sigma)U \Big(a e_{\rootb_1} + \sum_{v \in V^{(1)}_{\prec}} b_{v}\Big) \notag \\ &=& U\int_{\sigma}\Big(a w_0(z) e_{\rootb_1} + \sum_{v \in V^{(1)}_{\prec}} w_{n_v+1}(z)b_v \Big)dA(z). 
\eeq
Since $Ue_{\rootb_1} \in \ker (\mathscr M^{(2)}_{z, q})^*,$ there exist $a' \in \mathbb C$, and $b'_{v} \in l^2(\mathsf{Chi}(v)) \ominus \langle \lambdab^v\rangle$ for $v \in V^{(2)}_{\prec}$
such that
\beqn
Ue_{\rootb_1} = a'e_{\rootb_2} + \sum_{v \in V^{(2)}_{\prec}}b'_v.
\eeqn
Letting $a=1$ and $b_v=0$ for all $v \in V^{(1)}_{\prec}$ in \eqref{eqn}, we obtain 
\beqn
\Big(\int_{\sigma}w_0(z) dA(z)\Big)\Big(a'e_{\rootb_2} + \sum_{v \in V^{(2)}_{\prec}}b'_v\Big) &=&
\Big(\int_{\sigma}w_0(z) dA(z)\Big)Ue_{\rootb_1} \\ &\overset{\eqref{eqn}}=& \nu_{{}_{\mathscr T_2}}(\sigma)U e_{\rootb_1} 
%\\ &=& \int_{\sigma}\Big(a' w_0(z) e_{\rootb_2} dA(z) \Big) \\ &+& \sum_{v \in V^{(2)}_{\prec}} \Big(\int_{\mathbb D}w_{n_v+1}(z)b'_v \Big)dA(z) 
\\
&\overset{\eqref{measure-0}}=& \int_{\sigma}\Big(a' w_0(z) dA(z)\Big)e_{\rootb_2} 
\\ &+& \sum_{v \in V^{(2)}_{\prec}} \int_{\sigma}\Big(w_{n_v+1}(z)dA(z)\Big)b'_v
\eeqn
for every Borel subset $\sigma$ of $\mathbb D$.
Since for some Borel set $\sigma,$ $\int_{\sigma}w_0(z) dA(z) \neq \int_{\sigma}w_l(z) dA(z)$ for any positive integer $l,$ by
comparing coefficients on both sides, we obtain
$
 b'_v = 0~\mbox{for every}~v \in V^{(2)}_{\prec},$ and hence $Ue_{\rootb_1} = a'e_{\rootb_2}$ with $|a'|=1.$ This shows that
 \beqn 
 U(\langle e_{\rootb_1} \rangle) =\langle e_{\rootb_2} \rangle.
 \eeqn
Since $U$ maps $\ker (\mathscr M^{(1)}_{z, q})^*$ unitarily onto $\ker (\mathscr M^{(2)}_{z, q})^*$, we obtain
\beq \label{U-r-r}
 U \Big(\bigoplus_{v \in V^{(1)}_{\prec}} l^2(\mathsf{Chi}(v)) \ominus \langle \lambdab^v\rangle) \Big)= \bigoplus_{v \in V^{(2)}_{\prec}} l^2(\mathsf{Chi}(v)) \ominus \langle \lambdab^v\rangle.
\eeq 
Next,
since $Ub_{v_0} \in \ker (\mathscr M^{(2)}_{z, q})^*$ for a fixed $v_0 \in V^{(1)}_{\prec},$
there exist  $b'_{v} \in l^2(\mathsf{Chi}(v)) \ominus \langle \lambdab^v\rangle$ for $v \in V^{(2)}_{\prec}$
such that
\beqn
Ub_{v_0} =  \sum_{v \in V^{(2)}_{\prec}}b'_v.
\eeqn
As in the preceding paragraph, one may let $a=0$ and $b_v=0$ for all $v \in V^{(1)}_{\prec} \setminus \{v_0\}$ in \eqref{eqn} to obtain 
\beqn
\Big(\int_{\sigma}w_{n_{v_0}+1}(z) dA(z)\Big)\Big(\sum_{v \in V^{(2)}_{\prec}}b'_v\Big) 
%&=&
%\Big(\int_{\sigma} w_{n_{v_0}+1}(z) dA(z)\Big)Ub_{v_0} \\ &=& \nu_{{}_{\mathscr T_2}}(\sigma)Ub_{v_0} \\ 
&=&  \sum_{v \in V^{(2)}_{\prec}} \int_{\sigma}\Big(w_{n_v+1}(z)dA(z)\Big)b'_v. 
\eeqn
Since for some Borel set $\sigma,$ $\int_{\sigma}w_l(z) dA(z) \neq \int_{\sigma}w_m(z) dA(z)$ for positive integers $l \neq m,$ by
comparing coefficients on both sides, we obtain
$
 b'_v = 0~\mbox{for every}~v \in V^{(2)}_{\prec} \cap \mathcal G^{(2)}_n$ with $n \neq n_{v_0},$ 
and hence $Ub_{v_0} =\sum_{v \in V^{(2)}_{\prec} \cap \mathcal G_{n_{v_0}}}b'_v .$
Since $v_0$ is arbitrary, by \eqref{U-r-r}, 
$U$ must map $W^{(n)}_1$ injectively onto $W^{(n)}_2$ for every $n \in \mathbb N$, where
$$W^{(n)}_j=\bigoplus_{v \in V^{(j)}_{\prec} \cap \mathcal G^{(j)}_{n}} l^2(\mathsf{Chi}(v)) \ominus \langle \lambdab^v\rangle, \quad n \in \mathbb N, ~j=1, 2.$$
Thus for every $n \in \mathbb N,$
\beqn
\sum_{v \in V^{(1)}_{\prec} \cap \mathcal G^{(1)}_{n}} \Big(\mbox{card}(\mathsf{Chi}(v))-1 \Big)= \sum_{v \in V^{(2)}_{\prec} \cap \mathcal G^{(2)}_{n}}\Big(\mbox{card}(\mathsf{Chi}(v))-1\Big).
\eeqn

Conversely, suppose \eqref{constant-g2} holds for every $n \in \mathbb N.$ For $n \in \mathbb N$ and $j=1, 2$, let $$\mathcal M_{j, 0}=\langle e_{\rootb_j} \rangle, \quad \mathcal M_{j, n} =\bigoplus_{v \in V^{(1)}_{\prec} \cap \mathcal G^{(j)}_{n}} l^2(\mathsf{Chi}(v)) \ominus \langle \lambdab^v\rangle.$$
By \eqref{constant-g2}, we must have $\dim \mathcal M_{1, n}=\dim \mathcal M_{2, n}$ for every $n \in \mathbb N.$ Let $U = \oplus_{n \in \mathbb N}U_n$, where $U_n$ is a unitary from $\mathcal M_{1, n}$ onto $\mathcal M_{2, n}$ for $n \in \mathbb N.$ Further, we can choose $U_0$ such that $Ue_{\rootb_1}=e_{\rootb_2}$.
Note that $U$ is a unitary from 
$\ker (\mathscr M^{(1)}_{z, q})^*$ onto $\ker (\mathscr M^{(2)}_{z, q})^*$. We verify that $U$ satisfies \eqref{uni-E}. It suffices to check that 
\beq \label{uni-E-1}
U^*\nu_{{}_{\mathscr T_2}}(\sigma)Uf=\nu_{{}_{\mathscr T_1}}(\sigma)f,
\eeq
where $f=\alpha e_{\rootb_1} + \beta b_{v, n}$ for $\alpha, \beta \in \mathbb C$ and $\{b_{v, n}\}$ is an orthonormal basis for $\mathcal M_{1, n}$ for integers $n \ge 1.$
Indeed, for a Borel subset $\sigma$ of $\mathbb D,$
\beqn
U^*\nu_{{}_{\mathscr T_2}}(\sigma)Uf &=& U^*\nu_{{}_{\mathscr T_2}}(\sigma)(\alpha e_{\rootb_2} + \beta Ub_{v, n}) \\ &=& U^*\int_{\sigma}\Big(\alpha w_0(z) e_{\rootb_2} + \beta w_{n+1}(z)Ub_{v, n} \Big)dA(z) \\
&=& \int_{\sigma}\Big(\alpha w_0(z) e_{\rootb_1} + \beta w_{n+1}(z)b_{v, n} \Big)dA(z) \\ 
&=& \nu_{{}_{\mathscr T_1}}(\sigma)f,
\eeqn
which completes the verification of \eqref{uni-E-1}.
One can now define $\tilde{U} : L^2_a(\nu_{{}_{\mathscr T_1}}) \rar L^2_a(\nu_{{}_{\mathscr T_2}})$ by \beq \label{unitary} (\tilde{U}f)(z)=U(f(z))~(f \in L^2_a(\nu_{{}_{\mathscr T_1}}), ~z \in \mathbb D).\eeq 
It is easy to see using \eqref{uni-E} that $\tilde{U}$ is a unitary map such that $\tilde{U}\mathscr M^{(1)}_{z, q} = \mathscr M^{(2)}_{z, q}\tilde{U}.$
\end{proof}

%It is interesting to know whether the conclusion of Theorem \ref{solution} holds for non-integral values of $q > 1$.

\medskip \textit{Acknowledgment}. \
The authors convey sincere thanks to Gadadhar Misra for some fruitful conversations pertaining to the theory of homogeneous operators.
In particular, they acknowledges his help for 
providing an exact reference for a fact essential in the proof of Proposition \ref{homogeneous}. Further, the last author expresses his gratitude to the Department of Mathematics, IIT Kanpur for their warm hospitality during the preparation of this paper.

\end{document}